\documentclass[a4paper, 11pt]{amsart}   
\usepackage{mathptmx, amssymb,amscd,latexsym, eulervm}
\usepackage{amsmath}
\usepackage{amsthm}
\usepackage{mathdots}
\usepackage[pagebackref,colorlinks=true,linkcolor=blue,urlcolor=blue]{hyperref}
\usepackage{color}
\usepackage[onehalfspacing]{setspace}
\usepackage{tabularx}
\usepackage{amsfonts}
\usepackage{paralist}
\usepackage{aliascnt}
\usepackage[initials, lite]{amsrefs}
\usepackage{amscd}
\usepackage{blkarray}
\usepackage{mathbbol}
\usepackage{setspace}
\usepackage[inner=2.5cm,outer=2.5cm, bottom=2.8cm]{geometry}
\usepackage{tikz, tikz-cd}
\usepackage{float}
\usepackage{booktabs}
\usepackage{tikz}
\usetikzlibrary{matrix}
\usetikzlibrary{arrows,calc}
\allowdisplaybreaks
\BibSpec{collection.article}{%
	+{}  {\PrintAuthors}                {author}
	+{,} { \textit}                     {title}
	+{.} { }                            {part}
	+{:} { \textit}                     {subtitle}
	+{,} { \PrintContributions}         {contribution}
	+{,} { \PrintConference}            {conference}
	+{}  {\PrintBook}                   {book}
	+{,} { }                            {booktitle}
	+{,} { }                            {series}
	+{, vol.} { }                            {volume}
	+{,} { }                            {publisher}
	+{,} { \PrintDateB}                 {date}
	+{,} { pp.~}                        {pages}
	+{,} { }                            {status}
	+{,} { \PrintDOI}                   {doi}
	+{,} { available at \eprint}        {eprint}
	+{}  { \parenthesize}               {language}
	+{}  { \PrintTranslation}           {translation}
	+{;} { \PrintReprint}               {reprint}
	+{.} { }                            {note}
	+{.} {}                             {transition}
	+{}  {\SentenceSpace \PrintReviews} {review}
}
\newcommand{\RR}{\mathbb{R}}
\newcommand{\kk}{\mathbb{k}}

\newcommand{\bn}{{\normalfont\mathbf{n}}}
\newcommand{\bm}{{\normalfont\mathbf{m}}}
\newcommand{\ee}{{\normalfont\mathbf{e}}}
\newcommand{\bmm}{\mathbb{M}}
\newcommand{\fP}{\mathfrak{P}}
\newcommand{\GG}{\mathcal{G}}

\newcommand{\NN}{\normalfont\mathbb{N}}
\newcommand{\ZZ}{{\normalfont\mathbb{Z}}}

\newcommand{\PP}{{\normalfont\mathbb{P}}}

\newcommand{\xx}{\normalfont\mathbf{x}}
\newcommand{\yy}{\normalfont\mathbf{y}}

\newcommand{\dd}{{\normalfont\mathbf{d}}}

\newcommand{\mm}{{\normalfont\mathfrak{m}}}

\newcommand{\QQ}{\mathbb{Q}}
\newcommand{\pp}{\mathfrak{p}}

\newcommand{\bb}{\mathfrak{b}}

\newcommand{\nn}{{\normalfont\mathfrak{N}}}

\newcommand{\rank}{\normalfont\text{rank}}

\newcommand{\sat}{{\normalfont\text{sat}}}

\newcommand{\Ker}{\normalfont\text{Ker}}

\newcommand{\HT}{\normalfont\text{ht}}

\newcommand{\Supp}{\normalfont\text{Supp}}

\newcommand{\Rees}{\mathcal{R}}

\newcommand{\BB}{\mathcal{B}}

\newcommand{\DD}{\mathbb{D}}

\newcommand{\OO}{\mathcal{O}}

\newcommand{\FF}{\mathfrak{F}}
\newcommand{\Vol}{{\normalfont\text{Vol}}}
\newcommand{\MV}{{\normalfont\text{MV}}}
\newcommand{\HL}{\normalfont\text{H}_{\mm}}
\newcommand{\HH}{\normalfont\text{H}}

\newcommand{\gr}{\normalfont\text{gr}}

\newcommand{\AAA}{\mathfrak{A}}

\newcommand{\Proj}{\normalfont\text{Proj}}
\newcommand{\Spec}{{\normalfont\text{Spec}}}

\newcommand{\multProj}{\normalfont\text{MultiProj}}

\newcommand{\fib}{\mathfrak{F}}
\newcommand{\sfib}{\widetilde{\mathfrak{F}}}

\newcommand{\ttt}{\mathbf{t}}
\newcommand{\bK}{\mathbf{K}}
\newtheorem{theorem}{Theorem}[section]

\newaliascnt{headcor}{headthm}

\aliascntresetthe{headcor}

\newaliascnt{headconj}{headthm}

\aliascntresetthe{headconj}

\newaliascnt{corollary}{theorem}
\newtheorem{corollary}[corollary]{Corollary}
\aliascntresetthe{corollary}

\newaliascnt{claim}{theorem}

\aliascntresetthe{claim}

\newaliascnt{lemma}{theorem}
\newtheorem{lemma}[lemma]{Lemma}
\aliascntresetthe{lemma}

\newaliascnt{conjecture}{theorem}

\aliascntresetthe{conjecture}

\newaliascnt{proposition}{theorem}
\newtheorem{proposition}[proposition]{Proposition}
\aliascntresetthe{proposition}

\theoremstyle{definition}
\newaliascnt{definition}{theorem}
\newtheorem{definition}[definition]{Definition}
\aliascntresetthe{definition}

\newaliascnt{notation}{theorem}
\newtheorem{notation}[notation]{Notation}
\aliascntresetthe{notation}

\newaliascnt{example}{theorem}
\newtheorem{example}[example]{Example}
\aliascntresetthe{example}

\newaliascnt{examples}{theorem}

\aliascntresetthe{examples}

\newaliascnt{remark}{theorem}
\newtheorem{remark}[remark]{Remark}
\aliascntresetthe{remark}

\newaliascnt{question}{theorem}

\aliascntresetthe{question}

\newaliascnt{questions}{theorem}

\aliascntresetthe{questions}

\newaliascnt{problem}{theorem}

\aliascntresetthe{problem}

\newaliascnt{construction}{theorem}

\aliascntresetthe{construction}

\newaliascnt{setup}{theorem}
\newtheorem{setup}[setup]{Setup}
\aliascntresetthe{setup}

\newaliascnt{algorithm}{theorem}

\aliascntresetthe{algorithm}

\newaliascnt{observation}{theorem}

\aliascntresetthe{observation}

\newaliascnt{defprop}{theorem}

\aliascntresetthe{defprop}

\def\equationautorefname~#1\null{(#1)\null}
\def\sectionautorefname~#1\null{Section #1\null}
\def\subsectionautorefname~#1\null{\S #1\null}

\begin{document}

	% ---------------Tittle and presentation----------------------------
	\title{A study of nonlinear multiview varieties}
	
	\author{Yairon Cid-Ruiz}
	\address[Cid-Ruiz]{Department of Mathematics, KU Leuven, Celestijnenlaan 200B, 3001 Leuven, Belgium}
	\email{yairon.cidruiz@kuleuven.be}

	\author{Oliver Clarke}
	\address[Clarke]{Department of Mathematics: Algebra and Geometry, Ghent University, S25, 9000 Gent, Belgium}
	\email{oliver.clarke@bristol.ac.uk}
	
	\author{Fatemeh Mohammadi}
	\address[Mohammadi]{Department of Mathematics, KU Leuven, Celestijnenlaan 200B, Leuven, Belgium  and Department of Computer Science, KU Leuven, Celestijnenlaan 200A, 3001 Leuven, Belgium}
	\email{fatemeh.mohammadi@kuleuven.be}

	\begin{abstract}
		We study the nonlinear generalization of the classical multiview variety, which is a fundamental concept in computer vision. 
		In this paper, we take the first comprehensive step to develop the nonlinear analogue of  multiview varieties.
		To this end, we introduce a multigraded version of the saturated special fiber ring. 
		By applying this tool, we are able to compute the multidegrees of several families of nonlinear multiview varieties. 
	\end{abstract}
	
	
	
	\maketitle
	{
		\hypersetup{linkcolor=black}
		\setcounter{tocdepth}{1}
		\tableofcontents
	}
	%\tableofcontents
	
	%------------------------------------------
	
	\section{Introduction}

	% 	We study a generalization of multiview varieties. 
	% 	Our interest and philosophy is very much inspired by the advances in the recent area of \emph{nonlinear algebra} \cite{Bernd_Michalek2021invitation}  ({\color{red}\sc more blah blah to sell the paper....}).
	
	\noindent	\textbf{Motivation.}
	Let $\kk$ be a field.
	The classical \emph{multiview variety} is an important tool in computer vision, see \cite{aholt2013hilbert}.
	It is defined as the closure of the image $Y$ of a linear rational map $\GG : \PP_\kk^3 \dashrightarrow \PP_\kk^2 \times_\kk \cdots \times_\kk \PP_\kk^2$. 
	Conceptually, the image of $\GG$ is a sequence of  cameras that take $2$-dimensional images of $3$-dimensional space. 
	Naturally, properties about $Y$ are crucial to solving the problem of reconstructing a $3$-dimensional object from a collection of images.
	
	Since $\GG$ is linear, the defining ideal $J$ of $Y$ has a structure that is amenable to study via a number of different methods. 
	Firstly, $J$ is an example of a Cartwright-Sturmfels ideal, see \cite{aholt2013hilbert, conca2018cartwright, conca2021radical}. 
	And so, $J$ satisfies a number of pleasant properties, see \cite{conca2021radical}.
	Secondly, it is possible to give an explicit description of the generators of $J$, see \cite{agarwal2019ideals}. In particular, if the focal points of the cameras are distinct then $J$ is generated by certain determinants of degree $2$ and $3$, see \cite[Theorem~3.7]{agarwal2019ideals}. 
	Thirdly, one can study the structure of the coordinate ring of $Y$ in terms of products of linear ideals, see \cite{conca2019resolution}. 
	In particular, it is possible to give an explicit formula for the Betti numbers in terms of the polymatroidal data defined by the positions of the cameras. These results demonstrate connections between properties of $Y$ and properties of the base locus of $\GG$. 
	The goal and purpose of this paper is to extend such connections for when $\GG$ is nonlinear.
	Therefore, our approach embodies the general and fruitful philosophy of studying nonlinear algebra \cite{NONLINEAR_BOOK, breiding2021nonlinear}.
	
	\smallskip

	\noindent	\textbf{Previous Work.} Let us consider a rational map $\GG : \PP_\kk^3 \dashrightarrow \PP_\kk^2 \times_\kk \cdots \times_\kk \PP_\kk^2$, as above, which is not necessarily linear. Note that, none of the methods above can be extended to study this case. 
	In particular, it is difficult to determine the generators of $J$ and, in general, this ideal is not necessarily a Cartwright-Sturmfels ideal. 
	So, it is natural to ask how to compute its Hilbert function. 
	More specifically, it is important to study the \emph{multidegrees} of $Y$ as a subvariety of the target $\PP_\kk^2 \times_\kk \cdots \times_\kk \PP_\kk^2$ of $\GG$.
	Our approach to determine the multidegrees of $Y$ is based on the study of blow-up algebras and  syzygies of the ideals generated by the defining polynomials of the rational map $\GG$.
	
	A syzygy-based approach to study rational maps seems to have been initiated in \cite{HULEK_KATZ_SCHREYER_SYZ}, and it has now become an active and fruitful research area, see~e.g.~\cite{AB_INITIO, Simis_cremona,KPU_blowup_fibers,EISENBUD_ULRICH_ROW_IDEALS,Hassanzadeh_Simis_Cremona_Sat,SIMIS_RUSSO_BIRAT,EFFECTIVE_BIGRAD,SIMIS_PAN_JONQUIERES,HASSANZADEH_SIMIS_DEGREES, MULTPROJ, MULT_SAT_PERF_HT_2, SAT_FIB_GOR_3, MIXED_MULT, BOTBOL_ALICIA_RAT_SURF, BOTBOL_ALICIA_MAT, KPU_BIGRAD_STRUCT, KPU_GOR3, KPU_NORMAL_SCROLL, CARLOS_MONO, CARLOS_MONOID, CARLOS_MU2}.
	In the area of Computer Aided Geometric Design (CAGD), a somewhat similar story has emerged where the concept of syzygies is substituted by an equivalent one in $\mu$-bases, see \cite{COX_EQ_PARAM, cox1998moving_line}.
	In all these previous works, the main goal is to study a (singly projective) rational map of the form $\GG : \PP_\kk^r \dashrightarrow \PP_\kk^s$.
	To illustrate the general idea, we recall an instance where the degree of the syzygies of the base ideal completely determines the degree of the image.
	In \cite{COX_EQ_PARAM}, Cox considers a parametric surface $Y$ given by the image of $\GG : \PP_\kk^2 \dashrightarrow \PP_\kk^3$. 
	In this case, under some technical assumptions, the degree of $Y$ can be computed by an elementary symmetric polynomial evaluated at the degrees of the syzygies of the base ideal, see \cite[Proposition~5.3]{COX_EQ_PARAM}. 
	In \cite[Theorem~A]{MULT_SAT_PERF_HT_2}, this result is widely generalized using the saturated special fiber ring to the case of rational maps $\GG : \PP_\kk^r \dashrightarrow \PP_\kk^s$ where the base locus is given by a perfect ideal of height two. The \emph{saturated special fiber ring} has been recently used to successfully study rational maps in many different contexts, see \cite{MULT_SAT_PERF_HT_2, SAT_FIB_GOR_3, MIXED_MULT, DEGREE_SPECIALIZATION}. 
	It was introduced by Bus\'e -- Cid-Ruiz -- D'Andrea \cite{MULTPROJ} for studying the degree and birationality of rational maps. In this paper we define the multigraded version of the saturated special fiber ring, see \autoref{def_sat_special_fiber_ring}.
	
	\smallskip

	\noindent	\textbf{Main Results.} %In this paper, 
	We consider rational maps $\GG : \PP_\kk^r \dashrightarrow \PP_\kk^{m_1} \times_\kk \cdots \times_\kk \PP_\kk^{m_p}$ whose image can be thought of as the nonlinear generalization of a multiview variety. 
	Note that, the closure of the image $Y$ of $\GG$ naturally admits a $\ZZ^p$-grading, and so we will be concerned with the multidegrees of $Y$ instead of the degree.	
	Our first main result is \autoref{thm:main_result_sat_special_fiber_ring} which relates the mixed multiplicities, see \autoref{def_multdeg}, of the saturated special fiber ring with the multidegrees of $Y$. 
	In particular, if the special fiber ring is integrally closed, we obtain a criterion for birationality of $\GG$ in terms of the saturated special fiber ring.
	
	Our second main result is \autoref{thm:zero_base_locus_degree_formula} which gives an upper bound for the multidegrees of $Y$, and, if the base locus of $\GG$ is zero-dimensional, then it yields an exact formula for the multidegrees of $Y$ in terms of the mixed multiplicities of the base points. 
	This generalizes the well-known degree formula for rational maps with finite base locus to the multiprojective setting, see \cite[Theorem 2.5]{Laurent_Jouanolou_Closed_Image}, \cite[Theorem 6.6]{Sim_Ulr_Vasc_mult} and \cite[Theorem 3.3]{MULTPROJ}.
	Note that, if $\GG$ is a linear rational map, then $Y$ is multiplicity-free, see \autoref{thm:multiplicity_free_linear_case} which recovers \cite[Theorem~1.1]{li2013images} and  \cite[Theorem~3.9]{conca2019resolution}.
	
	Our third main result is \autoref{thm:multidegree_formula_perfect_ht2_gorenstein_ht3} which gives an explicit formula for the multidegrees of $Y$ in the case that the base ideals $I_1, \dots, I_{p-1}$ are zero-dimensional and $I_p$ is either a height two perfect ideal or a height three Gorenstein ideal. 
	In particular, this result generalizes \cite[Theorem~A]{MULT_SAT_PERF_HT_2} and \cite[Theorem~A]{SAT_FIB_GOR_3} to the multigraded setting.
	
	Lastly, for the case of monomial rational maps, \autoref{thm:monomial_degree_mixedVol} provides a closed formula for the multidegrees of $Y$ in terms of mixed volumes of some naturally constructed lattice polytopes.
	
	\smallskip

	\noindent	\textbf{Outline.} In \autoref{sec:notation_and_prelim}, we fix our notation and recall the definitions of multidegrees and mixed multiplicities, see \autoref{def_multdeg}. 
	In \autoref{sec:saturated_special_fiber_ring}, we begin by fixing \autoref{setup_rat_maps} which details the generality for which our main results hold and, following this, we recall some fundamental properties of the multigraded Rees algebra. In \autoref{subsec:saturated_special_fiber_ring}, we introduce our main tool for studying the multidegrees of the image of rational maps, namely the saturated special fiber ring, see \autoref{def_sat_special_fiber_ring}. 
	Using this tool, we state and prove our first main result: \autoref{thm:main_result_sat_special_fiber_ring}. In \autoref{subsec:degree_formula}, we define the mixed multiplicities of the base locus of $\GG$ and prove our second main result: \autoref{thm:zero_base_locus_degree_formula}, which allows us to compute the multidegrees of $Y$ in terms of the mixed multiplicities of the base locus of $\GG$.
	
	We proceed to apply \autoref{thm:zero_base_locus_degree_formula} to important families of examples. 
	In \autoref{sec:linear_rational_maps}, we show that our results recover previous results for linear rational maps, i.e.~for classical multiview varieties. 
	In \autoref{sec:perfect_ht2_gorenstein_ht3}, we prove our third main result, \autoref{thm:multidegree_formula_perfect_ht2_gorenstein_ht3}, which computes the multidegrees of $Y$ for a large family of rational maps extending previous work in \cite{MULT_SAT_PERF_HT_2, SAT_FIB_GOR_3} to the multigraded setting. 
	In \autoref{sec:monomial_rational_maps}, we consider monomial rational maps $\GG$, and we show that in this case, the multidegrees of the image can be computed in terms of the mixed volumes of certain lattice polytopes.
	
	In \autoref{sec:applications}, we further elaborate on the previous works, along with interpretations and applications related to our results.

	\medskip
	\noindent
	\textbf{Acknowledgments.} 
	Y.C.R. was partially supported by an FWO Postdoctoral Fellowship (1220122N). 
	F.M. was partially supported by KU Leuven iBOF/23/064, and FWO grants (G023721N, G0F5921N).

	\section{Notations and preliminaries}\label{sec:notation_and_prelim}
	We begin by fixing some common notation and recalling some preliminary results. Let $p \ge 1$ be a positive integer and, for each $1 \le i \le p$, let $\ee_i \in \NN^p$ be the $i$-th elementary vector $\ee_i=\left(0,\ldots,1,\ldots,0\right)$.
	For any $\bn = (n_1,\ldots,n_p) \in \ZZ^p$, we define its weight as $\lvert \bn \rvert:= n_1+\cdots+n_p$. 
	For two vectors $\bn = (n_1,\ldots,n_p) \in \NN^p$ and $\delta = (\delta_1,\ldots,\delta_p) \in \NN^p$, we denote the dot product by $\bn \cdot \delta := n_1\delta_1 + \cdots + n_p\delta_p$.
	If $\bn = (n_1,\ldots,n_p),\, \bm = (m_1,\ldots,m_p) \in \ZZ^p$ are two vectors, we write $\bn \ge \bm$ whenever $n_i \ge m_i$ for all $1 \le i \le p$, and $\bn > \bm$ whenever $n_j > m_j$ for all $1 \le j \le p$.
	We write $\mathbf{0} \in \NN^p$ for the zero vector $\mathbf{0}=(0,\ldots,0)$.  
	
	\subsection{Multigraded schemes and multidegrees}
	We fix the following setup during this subsection.
	
	\begin{setup}
		Let $\kk$ be a field and $S$ a finitely generated standard $\NN^p$-graded algebra, i.e. $\left[S\right]_{\mathbf{0}}=\kk$ and $S$ is finitely generated over $\kk$ by elements of degree $\ee_i$ with $1 \le i \le p$.
	\end{setup}
	
	Given a standard $\NN^p$-graded $\kk$-algebra $S$, we consider the corresponding multigraded scheme.
	
	\begin{definition}	
		The multiprojective scheme  $\multProj(S)$ is given by $
		\multProj(S) := \big\{ \fP \in \Spec(S) \mid \fP \text{ is $\NN^p$-graded and } \fP \not\supseteq \nn \big\},
		$
		and its scheme structure is obtained by using multi-homogeneous localizations, see e.g. \cite[\S 1]{HYRY_MULTIGRAD}. 
		The multigraded irrelevant ideal $\nn \subset S$ is given by $\nn := \left([S]_{\ee_1}\right) \cap \cdots \cap \left([S]_{\ee_p}\right)$.
	\end{definition}
	
	Let $P_S(\ttt)=P_S(t_1,\ldots,t_p) \in \QQ[\ttt]=\QQ[t_1,\ldots,t_p]$ be the \emph{(multigraded) Hilbert polynomial} of $S$, see e.g. \cite[Theorem 4.1]{HERMANN_MULTIGRAD} and \cite[Theorem 3.4]{MIXED_MULT}.
	Then, the degree of $P_S$ is equal to $r:=\dim(\multProj(S)) = \dim(S/(0:_S\nn^\infty))-p$ and 
	$
	P_S(\nu) = \dim_\kk\left([S]_\nu\right) 
	$
	for all $\nu \in \NN^p$ such that $\nu \gg \mathbf{0}$.
	Furthermore, if we write 
	\begin{equation}
		\label{eq_Hilb_poly}
		P_{S}(\ttt) = \sum_{d_1,\ldots,d_p \ge 0} e(d_1,\ldots,d_p)\binom{t_1+d_1}{d_1}\cdots \binom{t_p+d_p}{d_p},
	\end{equation}
	then $0 \le e(d_1,\ldots,d_p) \in \ZZ$ for all $d_1+\cdots+d_p = r$. 
	Interchangeably, $P_S(\ttt)$ is the Hilbert polynomial $P_X(\ttt)$ of the multiprojective scheme $X = \multProj(S)$.
	
	We embed $X = \multProj(S)$, as a closed subscheme, into the multiprojective space $\PP:=\PP_\kk^{m_1} \times_\kk \cdots \times_\kk \PP_\kk^{m_p}$.
	Under the notation of \autoref{eq_Hilb_poly}, we define the following invariants. 
	
	\begin{definition} \label{def_multdeg}
		Let $\dd = (d_1,\ldots,d_p) \in \NN^p$ with $\lvert\dd\rvert=r$. Then:
		\begin{enumerate}[(i)]
			\item $e(\dd,S) := e(d_1,\ldots,d_p)$ is the \textit{mixed multiplicity of $S$ of type $\dd$}.
			\item $\deg_\PP^\dd(X):=e(d_1,\ldots,d_p)$ is the \textit{multidegree of $X=\multProj(S)$ of type $\dd$ with respect to $\PP$}. 
		\end{enumerate} 
	\end{definition}
	
	Whenever the chosen embedding of $X$ into a product of projective spaces $\PP:=\PP_\kk^{m_1} \times_\kk \cdots \times_\kk \PP_\kk^{m_p}$ is clear from the context, we simply write $\deg^\dd(X)$ instead of $\deg_\PP^\dd(X)$.
	In classical geometrical terms, when $\kk$ is algebraically closed, $\deg_\PP^\dd(X)$ is also equal to the number of points (counting multiplicity) in the intersection of $X$ with the product $L_1 \times_\kk \cdots \times_\kk L_p  \subset \PP$, where $L_i \subseteq \PP_\kk^{m_i}$ is a general linear subspace of dimension $m_i-d_i$ for each $1 \le i \le p$, see \cite{VAN_DER_WAERDEN} and \cite[Theorem 4.7]{MIXED_MULT}.
	
	\subsection{Mixed multiplicities of ideals}
	
	Let us recall the definition of mixed multiplicities of ideals in the setting that we are considering.	For more details, we refer the reader to \cite{TRUNG_VERMA_SURVEY}. Let $R = \kk[x_0,\ldots,x_r]$ be a standard graded polynomial ring, $\mm = (x_0,\ldots,x_r)$ the graded irrelevant ideal, and  $I_1, \ldots, I_p \subset R$ a sequence of homogeneous ideals. 
	We consider the $\NN^{p+1}$-graded $\kk$-algebra 
	$$
	T(\mm \mid I_1, \ldots, I_p) \, := \, \bigoplus_{(n_0, n_1, \ldots, n_p) \in \NN^{p+1}}    \frac{\mm^{n_0} I_1^{n_1} \cdots I_p^{n_p}}{\mm^{n_0+1} I_1^{n_1} \cdots I_p^{n_p}}.
	$$
	The mixed multiplicities of the ideals $I_1,\ldots,I_p$ with respect to $\mm$ are defined as the mixed multiplicities of the algebra $T(\mm \mid I_1, \ldots, I_p)$.
	
	\begin{definition}
		Let $(d_0, \dd) = (d_0, d_1,\ldots,d_p) \in \NN^{p+1}$ with $d_0 + \lvert\dd\rvert=r$. 
		The \emph{mixed multiplicity of $I_1,\ldots,I_p$ with respect to $\mm$ of type $(d_0,\dd)$} is  $e_{(d_0,\dd)}\left(\mm \mid I_1, \ldots, I_p \right) := e\left( (d_0, \dd), \, T(\mm \mid I_1, \ldots, I_p) \right)$.
	\end{definition}

	\section{Our main tool: the saturated special fiber ring}\label{sec:saturated_special_fiber_ring}
	
	In this section, we develop and study the multigraded saturated special fiber ring. Originally, the saturated special fiber ring was introduced in \cite{MULTPROJ}, and has been successfully used to study several classes of rational maps \cite{MULT_SAT_PERF_HT_2, GEN_FREENESS_LOC_COHOM, MIXED_MULT, DEGREE_SPECIALIZATION, SAT_FIB_GOR_3, cid2019blow}.
	
	We begin by fixing the following setup, which is used throughout this section.  
	
	\begin{setup}
		\label{setup_rat_maps}
		Let $\kk$ be an arbitrary field, $R = \kk[x_0,\ldots,x_r]$ a standard graded polynomial ring and $\mm=(x_0,\ldots,x_r) \subset R$ the graded irrelevant ideal of $R$.
		For each $1 \le i \le p$, let $\{ f_{i,0}, f_{i,1}, \ldots, f_{i,m_i} \} \subset R$ be a set of forms of the same degree $\delta_i \ge 1$, $I_i = ( f_{i,0}, f_{i,1}, \ldots, f_{i,m_i} ) \subset R$ the ideal generated by these forms and $\GG_i : \PP_\kk^r \dashrightarrow \PP_\kk^{m_i}$ the rational map determined by $(\xx) = (x_0,\ldots,x_r) \mapsto (f_{i,0}(\xx), f_{i,1}(\xx), \ldots, f_{i,m_i}(\xx))$.
		Let 
		$$
		\GG = (\GG_1, \ldots, \GG_p) \;:\;  \PP_\kk^r \dashrightarrow \PP_\kk^{m_1}  \times_\kk \cdots \times_\kk \PP_\kk^{m_p}
		$$
		be the rational map given by the product of $\GG_1,\ldots,\GG_p$.
		Let $Y \subset \PP_\kk^{m_1} \times_\kk \cdots \times_\kk \PP_\kk^{m_p}$ be the closure of the image of $\GG$ and $\Gamma = \overline{\{(\xx, \GG(\xx))\}} \subset \PP_\kk^r \times_\kk \PP_\kk^{m_1}  \times_\kk \cdots \times_\kk \PP_\kk^{m_p}$ the closure of the graph of $\GG$.
		Let $\PP$ be the multiprojective space $\PP:= \PP_\kk^{m_1}  \times_\kk \cdots \times_\kk \PP_\kk^{m_p}$.
	\end{setup}
	
	With respect to the above setting, we call $Y \subset \PP_\kk^{m_1}  \times_\kk \cdots \times_\kk \PP_\kk^{m_p}$ the \emph{nonlinear multiview variety} associated to $\GG$.
	Note that, if $\delta_i = 1$ for all $i$, then we are in the classical case of \emph{multiview varieties}.
	
	\subsection{Rational maps and some related algebras}
	In this subsection, we recall some basic facts regarding rational maps.
	The \emph{(multigraded) Rees algebra} of the ideals $I_1, \ldots, I_p$ is given by
	$$
	\mathcal{R}(I_1,\ldots,I_p) := R[I_1t_1,\ldots, I_pt_p] = \bigoplus_{n_1 \ge 0,\ldots, n_p \ge 0} I_1^{n_1} \dots I_p^{n_p} t_1^{n_1} \dots t_p^{n_p} \; \subset \; R[t_1,\ldots,t_p].
	$$
	Let $\AAA$ be the standard $\NN^{p+1}$-graded polynomial ring 
	$$
	\AAA \, := \, R \otimes_\kk \kk[\yy_1] \otimes_\kk \cdots \otimes_\kk \kk[\yy_p] \, = \,  R \otimes_\kk \kk[y_{1,0}, y_{1,1},\ldots,y_{1,m_1}] \otimes_\kk \cdots \otimes_\kk \kk[y_{p,0}, y_{p,1},\ldots,y_{p,m_1}],
	$$
	where $\deg(x_i) = \ee_1  \in \NN^{p+1}$ for all $0 \le i \le r$, and $\deg(y_{i, j}) = \ee_{i+1} \in \NN^{p+1}$ for all $1 \le i \le p, 0 \le j \le m_i$.
	By setting $\deg(t_i)=-\delta_i\ee_1 + \ee_{i+1} \in \NN^{p+1}$ for all $1 \le i \le p$, the Rees algebra $\Rees(I_1,\ldots,I_p)$ inherits a natural $\NN^{p+1}$-graded structure.
	And so, we can present $\Rees(I_1,\ldots,I_p)$ as quotient of $\AAA$ via the following $\NN^{p+1}$-homogeneous surjective $R$-homomorphism
	$$
	\AAA \twoheadrightarrow \Rees(I_1,\ldots,I_p), \qquad y_{i,j} \mapsto f_{i,j}t_i.
	$$
	Thus, the $\NN^{p+1}$-graded structure of $\Rees(I_1,\ldots,I_p)$ is given by 
	$$
	\Rees(I_1,\ldots,I_p) \;=\; \bigoplus_{c \in \NN,\; \bn \in \NN^p} {\big[\Rees(I_1,\ldots,I_p)\big]}_{c, \bn}\ , 
	$$ 
	where each $\NN^{p+1}$-graded part is
	$$
	{\big[\Rees(I_1,\ldots,I_p)\big]}_{c, \bn} \;=\;  {\big[\mathbf{I}^\bn\big]}_{c + \bn \cdot \delta}\ttt^\bn \;=\; {\big[I_1^{n_1}\cdots I_p^{n_p}\big]}_{c + n_1\delta_1 + \cdots + n_p\delta_p} t_1^{n_1}\cdots t_p^{n_p}.
	$$
	
	Since we are primarily interested in the $R$-grading of the Rees algebra, we fix the following notation.
	
	\begin{notation}
		Let $\bmm$ be an $\NN^{p+1}$-graded module over $\AAA$.
		We denote by ${\left[\bmm\right]}_c$ the ``one-sided'' $R$-graded part
		$
		{\left[\bmm\right]}_c = \bigoplus_{\bn \in \ZZ^{p}
		} {\left[\bmm\right]}_{c,\bn}.
		$
		Then, ${\left[\bmm\right]}_{c}$ has a natural $\NN^p$-graded structure, and its $\bn$-th graded part is given by 
		$
		{\big[{\left[\bmm\right]}_c\big]}_\bn={\left[\bmm\right]}_{c,\bn}
		$
		where $\bn \in \NN^p$.
	\end{notation}

	The \emph{(multigraded) special fiber ring} of the ideals $I_1,\ldots,I_p$ is given by
	$$
	\mathfrak{F}(I_1,\ldots,I_p) \,:=\, \big[\Rees(I_1,\ldots,I_p)\big]_0 \,\cong\,  \kk\big[[I_1]_{\delta_1}t_1,\ldots, [I_p]_{\delta_p}t_p\big] \,\cong\, \bigoplus_{n_i \ge 0} \left[I_1^{n_1}\cdots I_p^{n_p}\right]_{n_1\delta_1 + \cdots + n_p\delta_p}.
	$$
	Note that $\mathfrak{F}(I_1,\ldots,I_p)$ is naturally a standard $\NN^p$-graded $\kk$-algebra.
	As the ideals $I_i$'s are equigenerated in degrees $\delta_i$'s,  Nakayama's lemma yields the natural isomorphism $\mathfrak{F}(I_1,\ldots,I_p) \cong \mathcal{R}(I_1,\ldots,I_p) \otimes_R R/\mathfrak{m}$.
	
	In the following proposition, we recall the known relationships between the rational map $\GG$ and the algebras $\Rees(I_1,\ldots,I_p)$ and $\FF(I_1,\ldots,I_p)$.
	
	\begin{proposition}
		\begin{enumerate}[(i)]
			\item The closure of the image of $\GG$ coincides with $Y = \multProj(\FF(I_1,\ldots,I_p))$.
			\item  The closure of the graph of $\GG$ coincides with $\Gamma = \multProj(\Rees(I_1,\ldots,I_p))$.
		\end{enumerate}
	\end{proposition}
	\begin{proof}
		For details, see e.g. \cite[\S 3]{DEGREE_SPECIALIZATION}.
	\end{proof}
	
	We will now recall some equivalent conditions for a rational map to be generically finite. 
	For any homogeneous ideal $J \subset R$, we denote by $\ell(J) := \dim\left(\Rees(J) / \mm \Rees(J)\right)$ the \emph{analytic spread} of $J$.
	We have the following equality 
	$$
	\dim(Y) = \dim\left(\kk\Big[[I_1]_{\delta_1}t_1,\ldots, [I_p]_{\delta_p}t_p\Big]\right)-p.
	$$
	The Segre embedding yields the isomorphism 
	$$
	Y = \multProj\left(\kk\Big[[I_1]_{\delta_1}t_1,\ldots, [I_p]_{\delta_p}t_p\Big]\right) \cong \Proj\left(\kk\Big[[I_1\cdots I_p]_{\delta_1+\cdots + \delta_p}\Big]\right).
	$$
	Thus, we obtain that $\dim(Y) = \ell(I_1\cdots I_p) -1$.
	Finally, we say that the rational map $\GG$ is \emph{generically finite} when any of the following equivalent conditions is satisfied: 
	\begin{enumerate}[(i)]
		\item $\left[K(\PP_\kk^r):K(Y)\right] < \infty$ where $K(\PP_\kk^r)$ and $K(Y)$ denote the fields of rational functions of $\PP_\kk^r$ and $Y$, respectively.
		\item $\dim(Y)=r$.
		\item $\ell(I_1\cdots I_p) -1 = r$.
	\end{enumerate} 
	
	Whenever $\GG$ is generically finite, we define the {\em degree} of $\GG:\PP_\kk^r \dashrightarrow \PP_\kk^{m_1}  \times_\kk \cdots \times_\kk \PP_\kk^{m_p}$ to be
	$$
	\deg(\GG):=\left[K(\PP_\kk^r):K(Y)\right].
	$$
	The map $\GG$ is said to be \emph{birational} if $\deg(\GG)=1$.
	If the degree of the field extension $K(\PP_\kk^r)|K(Y)$ is infinite, then we say that $\GG$ has no well-defined degree and, by convention, we write $\deg(\GG)=0$. 
	
	\subsection{A multigraded version of the saturated special fiber ring}\label{subsec:saturated_special_fiber_ring}
	
	Our main tool to study the rational map $\GG : \PP_\kk^r \dashrightarrow \PP_\kk^{m_1}  \times_\kk \cdots \times_\kk \PP_\kk^{m_p}$ is the following algebra.
	
	\begin{definition}\label{def_sat_special_fiber_ring}
		The \emph{(multigraded) saturated special fiber ring} of the ideals $I_1,\ldots,I_p$ is the $\NN^p$-graded $\kk$-algebra given by
		$$
		\sfib(I_1,\ldots,I_p) \, :=\, \bigoplus_{\bn \in \NN^p} \Big[\left(\mathbf{I}^\bn :_R \mm^\infty\right)\Big]_{\bn \cdot \delta} = \bigoplus_{n_1 \ge 0,\ldots, n_p \ge 0}  \Big[\left(I_1^{n_1}\cdots I_p^{n_p} :_R \mm^\infty\right)\Big]_{n_1\delta_1+\cdots+n_p\delta_p}.
		$$
	\end{definition}

	An important result regarding the saturated special fiber ring is the following.
	
	\begin{proposition}
		We have that $\sfib(I_1,\ldots,I_p)$ is a finitely generated $\NN^p$-graded module over the standard $\NN^p$-graded $\kk$-algebra $\FF(I_1,\ldots,I_p)$.
		In particular, $\sfib(I_1,\ldots,I_p)$ has a Hilbert polynomial $P_{\sfib(I_1,\ldots,I_p)}(\ttt)$.
	\end{proposition}
	\begin{proof}
		The proof is analogous to the proof of \cite[Proposition 2.7]{MULTPROJ}.
	\end{proof}
	
	Our main result is the following theorem, which says that the mixed multiplicities of the saturated special fiber ring are equal to the product of the multidegrees of $Y$ and the degree of $\GG$.
	This provides an extension \cite[Theorem 2.4]{MULTPROJ} for nonlinear multiview varieties.
	
	\begin{theorem}\label{thm:main_result_sat_special_fiber_ring}
		Assume \autoref{setup_rat_maps} and suppose that $\GG : \PP_\kk^r \dashrightarrow \PP_\kk^{m_1}  \times_\kk \cdots \times_\kk \PP_\kk^{m_p}$ is generically finite. 
		Then, the following statements hold:
		\begin{enumerate}[(i)]
			\item For all $\dd \in \NN^p$ with $\lvert \dd \rvert = r$, we have the equality
			$$
			e\big(\dd, \sfib(I_1,\ldots,I_p)\big) \; = \; \deg(\GG) \cdot \deg_\PP^\dd(Y).
			$$
			\item Under the additional condition of $\FF(I_1,\ldots,I_p)$ being integrally closed, $\GG$ is birational if and only if $\sfib(I_1,\ldots,I_p) = \FF(I_1,\ldots,I_p)$.
		\end{enumerate}
	\end{theorem}
	\begin{proof}
		(i) For ease of notation, let $S := \FF(I_1,\ldots,I_p)$ and $T := \sfib(I_1,\ldots,I_p)$.
		We consider the following functions
		$$
		F_S(\bn) = \lim_{n \rightarrow \infty} \frac{\dim_\kk\left(\left[S\right]_{n\cdot \bn}\right)}{n^r} \quad \text{ and } \quad F_T(\bn) = \lim_{n \rightarrow \infty} \frac{\dim_\kk\left(\left[T\right]_{n\cdot \bn}\right)}{n^r}.
		$$
		These functions coincide with the top-degree parts of the Hilbert polynomials of $S$ and $T$. 
		In other words, we can write
		$$
		F_S(\bn) = \sum_{\dd \in \NN^p, \, \lvert \dd \rvert = r} \frac{\deg^\dd(Y)}{\dd!} \, \bn^\dd \quad \text{ and } \quad F_T(\bn) = \sum_{\dd \in \NN^p, \, \lvert \dd \rvert = r} \frac{e\big(\dd, \sfib(I_1,\ldots,I_p)\big)}{\dd!} \, \bn^\dd .
		$$
		Since we have the equalities 
		$$
		S^{(\bn)} = \bigoplus_{n = 0}^\infty \left[I_1^{nn_1}\cdots I_p^{nn_p}\right]_{nn_1\delta_1 + \cdots + nn_p\delta_p} = \FF(I_1^{n_1} \cdots I_p^{n_p})
		$$ 
		and 
		$$
		T^{(\bn)} = \bigoplus_{n = 0}^\infty \left[\left(I_1^{nn_1}\cdots I_p^{nn_p}:_R \mm^\infty\right)\right]_{nn_1\delta_1 + \cdots + nn_p\delta_p} = \sfib(I_1^{n_1} \cdots I_p^{n_p}),
		$$
		it follows that the functions $F_S$ and $F_T$ can be obtained by computing the multiplicities of the special fiber ring and the saturated special fiber ring of products of powers of the ideals $I_i$'s.
		More precisely, we have that
		$$
		F_S(\bn) = e\left(\FF(I_1^{n_1}\cdots I_p^{n_p})\right) \quad \text{ and } \quad F_T(\bn) = e\left(\sfib(I_1^{n_1}\cdots I_p^{n_p})\right).
		$$
		For each $\bn \in \ZZ_+^r$, we compose Veronese and Segre embeddings to obtain the following natural embedding 
		$$
		\Delta : \PP_\kk^{m_1}  \times_\kk \cdots \times_\kk \PP_\kk^{m_p} \;\hookrightarrow\; \PP_\kk^{\binom{n_1+m_1}{m_1}-1} \times_\kk \cdots \times_\kk \PP_\kk^{\binom{n_p+m_p}{m_p}-1} \;\hookrightarrow\; \PP_\kk^m
		$$
		where $m = \binom{n_1+m_1}{m_1} \cdots \binom{n_p+m_p}{m_p} - 1$.
		Hence, we obtain the following commutative diagram 
		\begin{equation*}
			\label{eq:graph_Pi}
			\begin{tikzpicture}[baseline=(current  bounding  box.center)]
				\matrix (m) [matrix of math nodes,row sep=2.8em,column sep=13.5em,minimum width=2em, text height=1.5ex, text depth=0.25ex]
				{
					\PP_\kk^r & Y    \\
					& Y_\bn',  \\
				};
				\path[-stealth]
				(m-1-2) edge node [right]	{$\cong$}  (m-2-2)			
				;		
				\draw[->,dashed] (m-1-1)--(m-1-2) node [midway,above] {$\GG$};	
				\draw[->,dashed] (m-1-1)--(m-2-2) node [midway,above] {$\GG_\bn'$};	
			\end{tikzpicture}	
		\end{equation*}
		where $Y_\bn' \subset \PP_\kk^m$ is the image of $Y \subset  \PP_\kk^{m_1}  \times_\kk \cdots \times_\kk \PP_\kk^{m_p}$ under the embedding $\Delta$ and $\GG_\bn' : \PP_\kk^r \dashrightarrow Y_\bn' \subset \PP_\kk^m$ is the rational map induced by the set of generators naturally obtained for the ideal $I_1^{n_1}\cdots I_p^{n_p}$.
		It is clear that $\deg(\GG) = \deg(\GG_\bn')$.
		So, for each $\bn \in \ZZ_+^p$, \cite[Theorem 2.4]{MULTPROJ} may be applied to the rational map $\GG_\bn'$ and gives the following equalities 
		$$
		F_T(\bn) = e\left(\sfib(I_1^{n_1}\cdots I_p^{n_p})\right) = \deg(\GG') \cdot \deg(Y_\bn') = \deg(\GG) \cdot e\left(\FF(I_1^{n_1}\cdots I_p^{n_p})\right) = \deg(\GG) \cdot F_S(\bn).
		$$
		Finally, by comparing the coefficients of the polynomials $F_S(\bn)$ and $F_T(\bn)$, the claimed equality 
		$$
		e\big(\dd, \sfib(I_1,\ldots,I_p)\big) \; = \; \deg(\GG) \cdot \deg^\dd(Y)
		$$ 
		follows for all $\dd \in \NN^p$ with $\lvert \dd \rvert = r$.
		
		(ii) The proof follows similarly to the proof of \cite[Theorem 2.4(iv)]{MULTPROJ}.
	\end{proof}

	\subsection{A cohomological formula relating the degree of the rational map and the multidegrees of the image}
	
	Throughout this subsection, we continue using \autoref{setup_rat_maps}.	
	Here, we give a formula that relates the degree of $\GG$ with the mixed multiplicities of ${\left[\HL^1(\Rees(I_1,\ldots,I_p))\right]}_{0}$ and the multidegrees of the image $Y$.
	Notice that ${\left[\HL^1(\Rees(I_1,\ldots,I_p))\right]}_{0}$ is a finitely generated $\NN^p$-graded module over the special fiber ring $\FF(I_1,\ldots,I_p)$, and so we can consider its mixed multiplicities.
	The following result extends \cite[Corollary 2.12]{MULTPROJ} into the current setting.

	\begin{corollary}
		Assume \autoref{setup_rat_maps} and suppose that $\GG : \PP_\kk^r \dashrightarrow \PP_\kk^{m_1}  \times_\kk \cdots \times_\kk \PP_\kk^{m_p}$ is generically finite. 
		Then, for all $\dd \in \NN^p$ with $\lvert \dd \rvert = r$, we have the equality 
		$$
		\deg^\dd(Y) \cdot \big(\deg(\GG) - 1 \big) \, = \, e\left(\dd, {\left[\HL^1(\Rees(I_1,\ldots,I_p))\right]}_{0}\right).
		$$
		In particular, $\GG$ is birational if and only if $\dim \left( \Supp\left({\left[\HL^1(\Rees(I_1,\ldots,I_p))\right]}_{0}\right) \cap Y \right) < r$.
	\end{corollary}
	\begin{proof}
		We have the short exact sequence 
		$$
		0 \,\rightarrow\, \bigoplus_{\bn \in \NN^p} \big[\mathbf{I}^\bn \big]_{\bn \cdot \delta} \, \rightarrow \, \bigoplus_{\bn \in \NN^p} \Big[\left(\mathbf{I}^\bn :_R \mm^\infty\right)\Big]_{\bn \cdot \delta} \,\rightarrow\, \bigoplus_{\bn \in \NN^p} \Big[\HL^0\left(R/\mathbf{I}^\bn \right)\Big]_{\bn \cdot \delta} \, \rightarrow\, 0, 
		$$
		which can then be written as 
		$$
		0 \rightarrow \FF(I_1,\ldots,I_p) \rightarrow \sfib(I_1,\ldots,I_p) \rightarrow {\left[\HL^1(\Rees(I_1,\ldots,I_p))\right]}_{0} \rightarrow 0.
		$$
		The additivity of mixed multiplicities gives the equality 
		$$
		e(\dd, \sfib(I_1,\ldots,I_p))  = e(\dd,\FF(I_1,\ldots,I_p)) + e(\dd, {\left[\HL^1(\Rees(I_1,\ldots,I_p))\right]}_{0}),
		$$
		and so the result follows from \autoref{thm:main_result_sat_special_fiber_ring}.
	\end{proof}

	\subsection{A degree formula for rational maps}\label{subsec:degree_formula}
	Here, we study the case when $\GG : \PP_\kk^r \dashrightarrow \PP_\kk^{m_1}  \times_\kk \cdots \times_\kk \PP_\kk^{m_p}$ is not defined only in a finite number of points. 
	In particular, we provide a formula that relates the degree of $\GG$,  the multidegrees of the image, and the mixed multiplicities of the base points. 
	This can be seen as an extension of a known degree formula, see~e.g.~\cite[Theorem 2.5]{Laurent_Jouanolou_Closed_Image}, \cite[Theorem 6.6]{Sim_Ulr_Vasc_mult} and \cite[Theorem 3.3]{MULTPROJ}).
	This formula could also be derived by utilizing \cite[Proposition 4.4]{FULTON_INTERSECTION_THEORY}.
	We also give a general upper bound for the degree of the image for arbitrary rational maps that are generically finite.
	
	Hereafter, we use the same notations and conventions of \autoref{setup_rat_maps}.
	We have that the \emph{base locus of $\GG$} (i.e.~the points where $\GG$ is not well-defined) is given by 
	$$
	\BB(\GG) := V(I_1\cdots I_p) \,\subset\, \PP_\kk^r.
	$$	
	We then have that $\dim(\BB(\GG)) =0$ (i.e.~the base locus is \emph{finite}) if and only if $\dim(R/I_i) \le 1$ for all $1 \le i \le p$.
	Assuming that $\dim(\BB(\GG))=0$,
	we then have the equalities
	\begin{align*}
		\dim_{\kk}\Big(\HH^0\Big(\PP_\kk^r, \mathcal{O}_{\PP_\kk^r}/(I_1^{n_1}\cdots I_p^{n_p})^{\sim}\Big)\Big)  &= \sum_{\pp \in \BB(\GG)} \dim_{\kk}\Big({\big(\mathcal{O}_{\PP_\kk^r}/(I_1^{n_1}\cdots I_p^{n_p})^{\sim}\big)}_{\pp}\Big) \\
		&=\sum_{\pp \in \BB(\GG)} \left[\kappa(\pp):\kk\right]\cdot \text{length}_{\mathcal{O}_{\PP_\kk^r,\pp}}\Big({\big(\mathcal{O}_{\PP_\kk^r}/(I_1^{n_1}\cdots I_p^{n_p})^{\sim}\big)}_{\pp}\Big) \\
		&= \sum_{\pp \in \BB(\GG)} \left[\kappa(\pp):\kk\right]\cdot \text{length}_{R_{\pp}}\Big(R_{\pp}/(I_1^{n_1}\cdots I_p^{n_p})R_\pp\Big)
	\end{align*}
	for all $n_1,\ldots,n_p \in \NN$, where $\kappa(\pp)$ denotes the residue field of the local ring $\mathcal{O}_{\PP_\kk^r,\pp}$.
	Since we have $\dim(R_\pp/I_iR_\pp) = 0$ by assumption, it follows that the function 
	$$
	P_{R_\pp;I_1,\ldots,I_p}(n_1,\ldots,n_p) \;:=\; \lim\limits_{n\rightarrow \infty} \frac{\text{length}_{R_{\pp}}\Big(R_{\pp}/(I_1^{nn_1}\cdots I_p^{nn_p})R_\pp\Big)}{n^r} 
	$$ 
	is a polynomial in $\QQ[n_1,\ldots,n_p]$. For further details, we refer the reader to the survey \cite{TRUNG_VERMA_SURVEY} and Chapter 17 of the book \cite{huneke2006integral}.
	The polynomial $P_{R_\pp;I_1,\ldots,I_p}(n_1,\ldots,n_p)$ is of degree $r$ if there is some $i$ with $I_i R_\pp \neq \pp R_\pp$, and the zero polynomial otherwise.
	We write the polynomial $P_{R_\pp;I_1,\ldots,I_p}$ as 
	$$
	P_{R_\pp;I_1,\ldots,I_p}(n_1,\ldots,n_p) \,=\, \sum_{ \dd=(d_1,\ldots, d_p)\in \NN^{p},\,|\dd|=r}\frac{ e_{\dd}(R_\pp; I_1,\ldots, I_p)}{d_1!\cdots d_p!} n_1^{d_1}\cdots n_p^{d_p},
	$$
	where the numbers $e_{\dd}(R_\pp; I_1,\ldots, I_p)$ are nonnegative integers called the {\it mixed multiplicities} of $R_\pp$ with respect to $I_1,\ldots, I_p$.
	By summarizing the above results and discussion, it follows that the following expression
	\begin{equation*}
		P_{\BB(\GG)}(n_1,\ldots,n_p)
		\;:=\; \lim\limits_{n\rightarrow \infty} \frac{\dim_{\kk}\Big(\HH^0\Big(\PP_\kk^r, \mathcal{O}_{\PP_\kk^r}/(I_1^{nn_1}\cdots I_p^{nn_p})^{\sim}\Big) \Big)}{n^r}
	\end{equation*}
	is a polynomial in $\QQ[n_1,\ldots,n_p]$.
	We have that $P_{\BB(\GG)}$ is a polynomial of degree $r$ if $\BB(\GG) \neq \emptyset$, and the zero polynomial otherwise.
	Furthermore, we can write  $P_{\BB(\GG)}$ as 
	$$
	P_{\BB(\GG)}(n_1,\ldots,n_p) \,=\, \sum_{ \dd=(d_1,\ldots, d_p)\in \NN^{p},\,|\dd|=r}\frac{ e_{\dd}(\BB(\GG))}{d_1!\cdots d_p!} n_1^{d_1}\cdots n_p^{d_p},
	$$
	where the numbers $e_{\dd}(\BB(\GG))$ are nonnegative integers and given by 
	$$
	e_{\dd}(\BB(\GG)) \;:=\; \sum_{\pp \in \BB(\GG)} \left[\kappa(\pp):\kk\right]\cdot e_{\dd}(R_\pp; I_1,\ldots, I_p).
	$$
	So, the nonnegative integers $e_{\dd}(\BB(\GG))$ gather the mixed multiplicities of the base points, and we call them the \emph{mixed multiplicities of the base locus of $\GG$}.
	
	\medskip	We are now ready to state the main result of this subsection. 
	
	\begin{theorem}\label{thm:zero_base_locus_degree_formula}
		Assume \autoref{setup_rat_maps} and suppose that $\GG : \PP_\kk^r \dashrightarrow \PP_\kk^{m_1}  \times_\kk \cdots \times_\kk \PP_\kk^{m_p}$ is generically finite.
		Then, for all $\dd = (d_1,\ldots,d_p) \in \NN^p$ with $\lvert \dd \rvert = r$, the following statements hold: 
		\begin{enumerate}[\rm (i)]
			\item $\deg_\PP^\dd(Y)\deg(\GG) \, \le \, \delta_1^{d_1} \cdots \delta_p^{d_p}$.
			\item If $\dim(\BB(\GG)) = 0$, then we have the equality
			$$
			\delta_1^{d_1} \cdots \delta_p^{d_p} \, = \, \deg_\PP^\dd(Y)\deg(\GG) + e_\dd(\BB(\GG)).
			$$
		\end{enumerate}
	\end{theorem}
	\begin{proof}
		For any $\bn = (n_1,\ldots,n_p) \in \ZZ_+^p$, we have the exact sequence of sheaves
		$$
		0\rightarrow {(I_1^{n_1}\cdots I_p^{n_p})}^\sim(\bn\cdot\delta) \rightarrow \mathcal{O}_{\PP_\kk^r}(\bn\cdot\delta) \rightarrow \frac{\mathcal{O}_{\PP_\kk^r}}{{(I_1^{n_1}\cdots I_p^{n_p})}^\sim}(\bn\cdot\delta) \rightarrow 0,
		$$
		that yields the following exact sequence in cohomology
		$$
		0 \rightarrow \HH^0\big(\PP_\kk^r, (\mathbf{I}^\bn)^\sim(\bn \cdot \delta)\big) \rightarrow \HH^0\big(\PP_\kk^r, \mathcal{O}_{\PP_\kk^r}(\bn\cdot\delta)\big) \rightarrow \HH^0\big(\PP_\kk^r, \frac{\mathcal{O}_{\PP_\kk^r}}{(\mathbf{I}^\bn)^\sim}(\bn\cdot\delta)\big) \rightarrow \HH^1\big(\PP_\kk^r, (\mathbf{I}^\bn)^\sim(\bn \cdot \delta)\big) \rightarrow 0.
		$$
		Notice that $\dim_{\kk}(\HH^0(\PP_\kk^r, \mathcal{O}_{\PP_\kk^r}(\bn\cdot\delta))) = \binom{\bn \cdot \delta + r}{r}$ is a polynomial of degree $r$ in $\QQ[\bn]$.	
		By \cite[Theorem 4.4]{DEGREE_SPECIALIZATION}, we have that 
		$$
		\dim\left(\left[\HL^i(\Rees(I_1,\ldots,I_p))\right]_0\right) \le (\dim(R) + p) - i = r+1+p-i.
		$$
		This implies that for any $i\ge 1$ and $\bn = (n_1,\ldots,n_p) \in \ZZ_+^p$ with $n_i \gg 0$, the expression
		$$
		\dim_{\kk}\big(\HH^i\big(\PP_\kk^r,{(I_1^{n_1}\cdots I_p^{n_p})}^\sim(\bn\cdot\delta)\big)\big)=\dim_{\kk}\big( {\left[\HL^{i+1}(\Rees(I_1,\ldots,I_p))\right]}_{0, \bn} \big)
		$$
		is a polynomial in $\bn$ of degree strictly less than $r$.		
		So, for any $\bn = (n_1,\ldots,n_p) \in \ZZ_+^p$ we have the equalities
		\begin{align*}
			\lim_{n \rightarrow \infty} \frac{h^0\big(\PP_\kk^r, (\mathbf{I}^{n \cdot\bn})^\sim(n\cdot\bn \cdot \delta)\big) - h^1\big(\PP_\kk^r, (\mathbf{I}^{n \cdot\bn})^\sim(n\cdot\bn \cdot \delta)\big)}{n^r} = \lim_{n \rightarrow \infty} \frac{h^0\big(\PP_\kk^r, (\mathbf{I}^{n \cdot\bn})^\sim(n\cdot\bn \cdot \delta)\big)}{n^r} = F_T(\bn),
		\end{align*}
		where 
		$F_T(\bn) = \lim_{n \rightarrow \infty} \frac{\dim_\kk\left(\left[T\right]_{n\cdot \bn}\right)}{n^r}$ is the polynomial function that encodes the mixed multiplicities of the saturated special fiber ring $T = \sfib(I_1,\ldots,I_p)$.
		Since $F_T(\bn)$ is a homogeneous polynomial of degree $r$ in $\QQ[\bn]$, it  follows that the following function 
		$$
		Q(\bn) \,:=\, \lim_{n\rightarrow \infty} \frac{h^0\big(\PP_\kk^r, \,\mathcal{O}_{\PP_\kk^r} / (\mathbf{I}^{n \cdot\bn})^\sim
			(\bn\cdot\delta)\big)}{n^r} \,=\, \lim_{n\rightarrow \infty} \frac{\binom{n \cdot \bn \cdot \delta + r}{r}}{n^r} - F_T(\bn)
		$$
		is also a homogeneous polynomial of degree $r$ in $\QQ[\bn]$. 
		The fact that $Q(\bn) \ge 0$ for all $\bn \in \NN^r$ implies that the coefficients of $Q(\bn)$ are all nonnegative.
		Therefore, the equality $F_T(\bn)  + Q(\bn) \,  = \, \lim_{n\rightarrow \infty} \frac{\binom{n\cdot \bn \cdot \delta + r}{r}}{n^r}$ and \autoref{thm:main_result_sat_special_fiber_ring} yield the general upper bounds $\deg_\PP^\dd(Y)\deg(\GG) \, \le \, \delta_1^{d_1} \cdots \delta_p^{d_p}$. 
		This completes the proof of part (i). 
		
		Next, we assume that the base locus of $\GG$ is zero-dimensional.	
		Since $\dim(\BB(\GG))=0$, we obtain the equalities
		$$
		Q(\bn) \,:=\, \lim_{n\rightarrow \infty} \frac{h^0\big(\PP_\kk^r, \,\mathcal{O}_{\PP_\kk^r} / (\mathbf{I}^{n \cdot\bn})^\sim
			(\bn\cdot\delta)\big)}{n^r} 
		= 
		\lim_{n\rightarrow \infty} \frac{h^0\big(\PP_\kk^r,\, \mathcal{O}_{\PP_\kk^r}/(\mathbf{I}^{n \cdot\bn})^\sim\big)}{n^r}   
		= P_{\BB(\GG)}(\bn).
		$$
		So, combining the above expressions for $Q(\bn)$, we obtain the following equality of polynomials 
		$$
		F_T(\bn)  + P_{\BB(\GG)}(\bn) \,  = \, \lim_{n\rightarrow \infty} \frac{\binom{n\cdot \bn \cdot \delta + r}{r}}{n^r}.
		$$
		And so, the result of part (ii) follows from \autoref{thm:main_result_sat_special_fiber_ring}.
		This completes the proof of the theorem.
	\end{proof}
	
	As a direct consequence of the above theorem, we obtain the following result.
	\begin{corollary}
		Assume \autoref{setup_rat_maps} and suppose that all the ideals $I_i$ are $\mm$-primary {\rm(}i.e.~the rational map $\GG$ is actually a morphism $\GG : \PP_\kk^r \rightarrow \PP_\kk^{m_1}  \times_\kk \cdots \times_\kk \PP_\kk^{m_p}${\rm)}. 
		Then, for all $\dd = (d_1,\ldots,d_p) \in \NN^p$ with $\lvert \dd \rvert = r$, we have the equality 
		$$
		\deg_\PP^\dd(Y)\deg(\GG) \, = \, \delta_1^{d_1} \cdots \delta_p^{d_p}.
		$$
	\end{corollary}
	
	\section{Linear rational maps}\label{sec:linear_rational_maps}
	
	In this section, we revisit the case of linear multiview varieties that has been extensively studied. Here, we use the following setup. 
	
	\begin{setup}
		\label{setup_linear}
		Assume
		\autoref{setup_rat_maps} and let $\delta_i = 1$ for each $i$. That is, the forms $f_{i,j}$ defining the rational map $\GG : \PP_\kk^r \dashrightarrow \PP_\kk^{m_1}  \times_\kk \cdots \times_\kk \PP_\kk^{m_p}$ are linear. In addition, we will assume that the field is algebraically closed.
	\end{setup}  
	
	We start by recalling that, in this setting, the mixed multiplicities of the special fiber ring are either zero or one.
	
	\begin{theorem}[{\cite[Theorem~1.1]{li2013images} and  \cite[Theorem~3.9]{conca2019resolution}}]\label{thm:multiplicity_free_linear_case}
		Assume \autoref{setup_linear} and suppose that the map $\GG : \PP_\kk^r \dashrightarrow \PP_\kk^{m_1}  \times_\kk \cdots \times_\kk \PP_\kk^{m_p}$ is generically finite. 
		Then, the closure $Y$ of the image of $\GG$ is multiplicity-free. 
		That is,  $\deg_\PP^\dd(Y) \in \{0, 1\}$ for all $\dd \in \NN^p$ with $\lvert \dd \rvert = r$. 
	\end{theorem}
	\begin{proof}
		This result follows directly from \autoref{thm:zero_base_locus_degree_formula}(i).
	\end{proof}
	
	Our next observation is the following interesting corollary of \autoref{thm:main_result_sat_special_fiber_ring}.
	
	\begin{corollary}
		\label{cor_eq_sat_fib_fib}
		Assume \autoref{setup_linear} and suppose that $\GG : \PP_\kk^r \dashrightarrow \PP_\kk^{m_1}  \times_\kk \cdots \times_\kk \PP_\kk^{m_p}$ is generically finite. 
		Then, the saturated special fiber ring and the special fiber ring coincide. 
		Explicitly, for each $(n_1, \dots, n_p) \in \NN^p$, we have
		\[
		\left[ \left( 
		I_1^{n_1} \cdots I_p^{n_p} :_R \mm^{\infty}
		\right) \right]_{n_1 + n_2 + \dots + n_p}
		=
		\left[ \left( 
		I_1^{n_1} \cdots I_p^{n_p}
		\right) \right]_{n_1 + n_2 + \dots + n_p}.
		\]
	\end{corollary}
	
	\begin{proof}
		By \autoref{thm:multiplicity_free_linear_case} we have that $\FF(I_1, \dots, I_p)$ is multiplicity-free. 
		Hence, by a theorem of Brion, see \cite[Theorem~1.11]{conca2018cartwright} and \cite[Theorem~2]{brion2002multiplicity}, the special fiber ring $\FF(I_1, \dots, I_p)$ is integrally closed. We may therefore apply \autoref{thm:main_result_sat_special_fiber_ring}(ii) to deduce that $\FF(I_1, \dots, I_p) = \sfib(I_1, \dots, I_p)$ if and only if $\GG$ is birational. By assumption, all forms $f_{i,j}$ defining $\GG$ are linear and so the $\GG$ is birational.
	\end{proof}

	We now restrict our attention to the case where the base locus of $\GG$ is zero-dimensional. For each $1 \le i \le p$, the rational map $\GG_i : \PP_\kk^r \dashrightarrow \PP_\kk^{m_i}$ is then given by 
	$$
	(x_0, \dots, x_r) \mapsto A_i (x_0, \dots, x_r)^T
	\quad\text{for some matrix } A_i \in \kk^{(m_i + 1) \times (r+1)}.
	$$
	Under the assumption that $\dim(\BB(\GG)) =0$, we set-theoretically obtain $\BB(\GG) = \{ \Ker(A_i) \mid 1 \le i \le p \}$ and we have the equivalent conditions  $\dim(R/I_i) \le 1$ and $\rank(A_i) \ge r$.
	We note that each ideal $I_i$ is linear, hence prime. 
	
	\begin{remark}
		Suppose that $r = 3$ and $m_i = 2$ for all $i$. The assumption that the base locus is zero-dimensional is equivalent to the defining matrices $A_i$ of the rational map $\GG_i$
		be of rank $3$. 
		In particular, this is the case studied in \cite{agarwal2019ideals} which we generalize in \autoref{cor:dim_zero_base_linear}.
	\end{remark}

	By applying \autoref{thm:zero_base_locus_degree_formula}, we can give an explicit description of the degrees $\dd$, where $|\dd| = r$, such that $\deg_\PP^\dd(Y) = 1$.

	\begin{corollary}\label{cor:dim_zero_base_linear}
		Assume \autoref{setup_linear} and suppose that $\dim(\BB(\GG))=0$.
		Let $\dd = (d_1, \dots, d_p) \in \NN^p$ with $|\dd| = r$.
		Then $\deg_\PP^{\dd}(Y) = 1$ if and only if for each point $\pp \in \BB(\GG)$ there exists $1 \le i \le p$ such that $I_i \neq \pp$ and $d_i \ge 1$.
	\end{corollary}
	
	\begin{proof}
		Fix a point $\pp \in \BB(\GG)$ in the base locus. We write $I_{i_1} = I_{i_2} = \dots = I_{i_t} = \pp$ for all of the ideals which correspond to the same point in the base locus.  We proceed by calculating the value of $e_{\dd}(R_{\pp} ; I_1, \dots, I_p)$. Note that for each $1 \le j \le p$ we have
		$$
		I_j R_\pp = \left\{ 
		\begin{array}{ll}
			\pp R_\pp & \text{if } j \in \{i_1, i_2, \dots, i_t\}, \\
			R_p & \text{otherwise}.
		\end{array}
		\right.
		$$
		So, it follows that
		$$
		e_{\dd}(R_{\pp} ; I_1, \dots, I_p) = \begin{cases}
			e(\pp R_\pp) = 1 \quad \text{if } d_{i_1}+d_{i_2}+\cdots+d_{i_t} = r  \\
			0 \quad\qquad\qquad \text{otherwise},
		\end{cases}
		$$
		for more details see \cite[\S 17.4]{huneke2006integral}.
		
		By \autoref{thm:zero_base_locus_degree_formula}, we have
		$
		1 = \deg_\PP^{\dd}(Y)\deg(\GG) + e_{\dd}(\BB(\GG)).
		$
		The mixed multiplicities of the base locus are given by $\sum [\kappa(\pp) : \kk ] e_{\dd}(R_\pp; I_1, \dots, I_p)$ where the sum is taken over the set of primes $\pp \in \BB(\GG)$.
		Since, $\kk$ is algebraically closed, for all $\pp \in \BB(\GG)$, we have $[\kappa(\pp) : \kk]=1$. 
		Finally, it follows that $\deg_\PP^\dd(Y) = 1$ if and only if for each point $\pp \in \BB(\GG)$ there exists some $1 \le i \le p$ such that $I_i \neq \pp$ and  $d_i \ge 1$.
	\end{proof}

	\section{Certain special families of rational maps}\label{sec:perfect_ht2_gorenstein_ht3}
	
	In this section, we consider some special families of rational maps for which we give exact formulas for the product of the degree of the rational map with the multidegrees of its image. 
	We shall extend the computations obtained in \cite{MULT_SAT_PERF_HT_2, SAT_FIB_GOR_3} for perfect ideals of height two and Gorenstein ideals of height three. 
	Throughout this section we continue using \autoref{setup_rat_maps}.

	We begin by studying the multidegrees of the graph $\Gamma \subset \PP_\kk^r \times_\kk \PP_\kk^{m_1}  \times_\kk \cdots \times_\kk \PP_\kk^{m_p}$ of the rational map $\GG : \PP_\kk^r \dashrightarrow \PP_\kk^{m_1}  \times_\kk \cdots \times_\kk \PP_\kk^{m_p}$.
	To simplify notation, let 
	$$
	\DD := \PP_\kk^r \times_\kk \PP_\kk^{m_1}  \times_\kk \cdots \times_\kk \PP_\kk^{m_p} = \PP_\kk^r \times_\kk \PP.
	$$
	Then, for any $(d_0, \dd)\in \NN^{p+1}$ with $d_0 + |\dd| = r$, we shall consider the multidegrees $\deg_\DD^{(d_0,\dd)}(\Gamma) = e\big((d_0,\dd),\, \Rees(I_1,\ldots,I_p)\big)$.
	We proceed by relating the multidegrees of the graph $\Gamma$ with the mixed multiplicities of $\sfib(I_1,\ldots,I_p)$. The following result provides an extension of \cite[Theorem 5.4]{MIXED_MULT} to our setting.
	
	\begin{theorem}
		\label{thm_deg_graph_sat_fib}
		Assume \autoref{setup_rat_maps}.
		Then, for all $\dd = (d_1,\ldots,d_p) \in \NN^{p}$ with $|\dd| = r$, we have 
		$$
		\deg_\DD^{(0,\dd)}(\Gamma) \, = \, e\big(\dd, \, \sfib(I_1,\ldots,I_p)\big).
		$$
	\end{theorem}
	\begin{proof}
		This proof follows similarly to \cite[Theorem 5.4]{MIXED_MULT}.
		Let $\bb = (\yy_1) \cap \cdots\cap  (\yy_p)$, $\nn = \mm \cap \bb$ and $\mathfrak{M} = \mm + \nn$. Note that $\nn$ is the irrelevant ideal in $\DD = \multProj(\AAA)$.
		From the Mayer-Vietoris exact sequence 
		$$
		\HH_{\mathfrak{M}}^i(\Rees(I_1,\ldots,I_p)) \,\rightarrow\, \HH_\mm^i(\Rees(I_1,\ldots,I_p)) \oplus \HH_\mathfrak{b}^i(\Rees(I_1,\ldots,I_p)) \,\rightarrow\, \HH_\nn^i(\Rees(I_1,\ldots,I_p)) \,\rightarrow\, \HH_\mathfrak{M}^{i+1}(\Rees(I_1,\ldots,I_p))
		$$
		and the fact that $\left[\HH_{\mathfrak{M}}^i(\Rees(I_1,\ldots,I_p))\right]_{0,\bn} = \left[\HH_{\bb}^i(\Rees(I_1,\ldots,I_p))\right]_{0,\bn} = 0$ for all $\bn \gg \mathbf{0}$, we obtain the isomorphism 
		\[
		\left[\HH_{\mm}^i(\Rees(I_1,\ldots,I_p))\right]_{0,\bn} \cong \left[\HH_{\nn}^i(\Rees(I_1,\ldots,I_p))\right]_{0,\bn} \quad \text{for all } \bn \gg \mathbf{0}.
		\]
		
		Let $X=\Proj_{R\text{-gr}}\left(\Rees(I_1,\ldots,I_p)\right)$ be the projective scheme obtained by considering $\Rees(I)$ as a single-graded $\kk$-algebra with the grading of $R$. 
		By using the relations between sheaf and local cohomologies, see e.g.~\cite[Corollary 1.5]{HYRY_MULTIGRAD} and \cite[Appendix A4.1]{EISEN_COMM}, we have that 
		$$
		\left[\sfib(I_1,\ldots,I_p)\right]_\bn \;\cong\; \left[\HH^0(X, \OO_X)\right]_\bn \;\cong\;  \HH^0(\Gamma, \OO_\Gamma(0,\bn)) \quad \text{for all } \bn \gg \mathbf{0}.
		$$
		The multigraded Hilbert polynomial of $\Rees(I_1,\ldots,I_p)$ is given by 
		$$
		P_{\Rees(I_1,\ldots,I_p)}(c,\bn) \,=\, \sum_{i\ge0} {(-1)}^i \dim_\kk\Big(\HH^i\left(\Gamma,\OO_\Gamma(c,\bn)\right)\Big),
		$$
		see e.g. \cite[Lemma 4.3]{Kleiman_geom_mult}.
		By \cite[Theorem 4.4]{DEGREE_SPECIALIZATION}, we have that $\left[\HH_\mm^i(\Rees(I_1,\ldots,I_p))\right]_0$ is a finitely generated $\NN^p$-graded module over $\fib(I_1,\ldots,I_p)$ such that  
		$$
		\dim\left(\left[\HH_\mm^i(\Rees(I_1,\ldots,I_p))\right]_0\right) \, \le \, (\dim(R) + p) - i = r+1+p-i.
		$$
		Hence, the fact that $\HH^i\left(\Gamma,\OO_\Gamma(0,\bn)\right) \cong \left[\HH_{\nn}^{i+1}(\Rees(I_1,\ldots,I_p))\right]_{0,\bn}\cong \left[\HH_{\mm}^{i+1}(\Rees(I_1,\ldots,I_p))\right]_{0,\bn}$ for all $i \ge 1$ and $\bn \gg \mathbf{0}$ gives the equality 
		$$
		\lim_{n \rightarrow \infty} \frac{P_{\Rees(I_1,\ldots,I_p)}(0,n\cdot\bn)}{n^r} \; = \; \lim_{n \rightarrow \infty} \frac{\dim_{\kk}\left(
			\left[\sfib(I_1,\ldots,I_p)\right]_{n \cdot \bn}
			\right)}{n^r}
		$$ 
		for all $\bn \in \ZZ_+^p$.
		So, the equality $
		\deg_\DD^{(0,\dd)}(\Gamma) \, = \, e\big(\dd, \, \sfib(I_1,\ldots,I_p)\big)
		$
		holds for all $\dd \in \NN^p$ with $|\bn| = r$.
	\end{proof}

	We also relate the multidegrees of the graph $\Gamma$ with the mixed multiplicities of the ideals $I_1,\ldots,I_p$ with respect to $\mm$.
	
	\begin{lemma}
		\label{lem_multdeg_mixed_mult}
		Assume \autoref{setup_rat_maps}. Then, for all $(d_0,\dd) 
		%= (d_0,d_1,\ldots,d_p) 
		\in \NN^{p+1}$ with $d_0 + |\dd| = r$, we have 
		$$
		\deg_\DD^{(d_0,\dd)}(\Gamma) \, = \, e_{(d_0,\dd)}\left( \mm \mid I_1, \ldots, I_p \right).
		$$		
	\end{lemma}	
	\begin{proof}
		Consider the $\NN^{p+1}$-graded $\kk$-algebra $T(\mm \mid I_1, \ldots, I_p)$.
		Since the ideals $I_1,\ldots,I_p$ are equally generated in degrees $\delta_1, \ldots,\delta_p$, by Nakayama's lemma we obtain 
		\begin{align*}
			\dim_{\kk}\left(\left[T(\mm \mid I_1, \ldots, I_p)\right]_{n_0, \bn}\right) & = 
			\dim_{\kk}\left(\frac{\mm^{n_0} I_1^{n_1} \cdots I_p^{n_p}}{\mm^{n_0+1} I_1^{n_1} \cdots I_p^{n_p}}\right)\\
			& = \dim_{\kk}\left(\left[\mm^{n_0} I_1^{n_1} \cdots I_p^{n_p}\right]_{n_0 + \bn \cdot \delta}\right) \\
			& = \dim_{\kk}\left(\left[I_1^{n_1} \cdots I_p^{n_p}\right]_{n_0 + \bn \cdot \delta}\right) \\
			& = \dim_{\kk}\left(\left[\Rees(I_1, \ldots, I_p)\right]_{n_0, \bn}\right)
		\end{align*}
		for all $(n_0, \bn) = (n_0,n_1, \ldots, n_p) \in \NN^{p+1}$.
		This shows that $T(\mm \mid I_1, \ldots, I_p)$ and $\Rees(I_1, \ldots, I_p)$ have the same Hilbert function, and so the result of the lemma follows.
	\end{proof}

	The following proposition deals with the process of cutting the graph $\Gamma$ with general hyperplanes in the target space $\PP$. Note that this is similar to \cite[Proposition 5.6]{MIXED_MULT}, where the case of cutting $\Gamma$ with general hyperplanes in the source $\PP_\kk^r$ was treated.

	\begin{proposition}
		\label{prop_cut_hyper}
		Assume \autoref{setup_rat_maps} with $\kk$ being an infinite field.
		Fix $1 \le i \le p$.
		Let $h = \alpha_0y_{i,0} + \alpha_1y_{i,1} + \cdots + \alpha_{m_i}y_{i,m_i}$ with $\alpha_i \in \kk$ such that $V(h) \subset \PP_\kk^{m_i}$ is a general hyperplane.
		Let $f = \alpha_0f_{i,0} + \alpha_1f_{i,1} + \cdots + \alpha_{m_i}f_{i,m_i} \in I_i$, $S = R/fR$ and $J_j = I_jS \subset S$ for all $1 \le j \le p$.
		Then, for all $(d_0, \dd) = (d_0,d_1,\ldots,d_p) \in \NN^{p+1}$ with $d_0 + |\dd| = r$ and $d_i \ge 1$, we have 
		$$
		\deg_\DD^{(d_0,\dd)}(\Gamma) \, = \, e\big((d_0,\ldots,d_i-1,\ldots,d_p), \Rees_S(J_1,\ldots,J_p)\big).
		$$
	\end{proposition}
	\begin{proof}	
		By using \cite[Lemma 3.7]{MIXED_MULT}, we choose $(\alpha_0,\ldots,\alpha_{m_i})$ in such a way that 
		$$
		h = \alpha_0y_{i,0} + \alpha_1y_{i,1} + \cdots + \alpha_{m_i}y_{i,m_i} \in \kk[\yy_i] \subset \AAA = R \otimes_\kk \kk[\yy_1] \otimes_\kk \cdots \otimes_\kk \kk[\yy_p]
		$$
		becomes a filter-regular element on $\Rees(I_1,\ldots,I_p)$ and on 
		$$
		\gr_{I_i}\big(\Rees(I_1,\ldots,\widehat{I_{i}},\ldots,I_p)\big) = \Rees(I_1,\ldots,I_p) \otimes_R R/I_i = \bigoplus_{n_1,\ldots,n_p \ge 0} \frac{I_1^{n_1}\cdots I_i^{n_i}\cdots I_p^{n_p}}{I_1^{n_1}\cdots I_i^{n_i+1}\cdots I_p^{n_p}}, 
		$$
		where $\Rees(I_1,\ldots,\widehat{I_{i}},\ldots,I_p) = \Rees(I_1,\ldots,I_{i-1},I_{i+1},\ldots,I_p)$.
		So,  by \cite[Lemma 3.9]{MIXED_MULT}, we have that 
		$$
		\deg_\DD^{(d_0,\dd)}(\Gamma) \, = \, e\big((d_0,\dd-\ee_i), \Rees(I_1,\ldots,I_p) / h\, \Rees(I_1,\ldots,I_p)\big)
		$$
		for all $(d_0,\dd) \in \NN^{p+1}$ with $d_0+|\dd| = r$ and $d_i \ge 1$.
		Notice that we have the following natural surjection
		$$
		\mathfrak{s} \,:\;  \frac{\Rees(I_1,\ldots,I_p)}{h\, \Rees(I_1,\ldots,I_p)} \cong \bigoplus_{\bn \in \NN^p} \frac{\mathbf{I}^\bn}{f\,\mathbf{I}^{\bn - \ee_i}} \ttt^\bn \;\; \twoheadrightarrow  \;\; \Rees_S(J_1,\ldots,J_p) \cong \bigoplus_{\bn \in \NN^p} \frac{\mathbf{I}^\bn + fR}{fR} \ttt^\bn. 
		$$
		For any $\pp \in \Spec(R) \setminus V(I_i)$, the localization  $\mathfrak{s} \otimes_R R_\pp$ of the surjection $\mathfrak{s}$ becomes an isomorphism. 
		Hence, there is some $l > 0$ such that $I_i^l \cdot \Ker(\mathfrak{s}) = 0$.
		Known dimension computations give us that $\dim\left(\Rees(I_1,\ldots,I_p)\right) =  r+1+p$ and $\dim\left(\gr_{I_i}\big(\Rees(I_1,\ldots,\widehat{I_{i}},\ldots,I_p)\big)\right) = r+p$, see e.g. \cite[\S 17.5]{huneke2006integral}.
		For any finitely generated $\NN^{p+1}$-graded $\AAA$-module $M$, we set $\Supp_{++}(M) := \Supp(M) \cap \multProj(\AAA)$.
		Since $h$ is a filter-regular element on $\gr_{I_i}\big(\Rees(I_1,\ldots,\widehat{I_{i}},\ldots,I_p)\big)$ and on $\Rees(I_1,\ldots,I_p)$, \cite[Lemma 3.9]{MIXED_MULT} and \cite[Lemma 1.2]{HYRY_MULTIGRAD} imply that 
		\begin{align*}
			\dim\left(\Supp_{++}\left(\Ker(\mathfrak{s})\right)\right) &\le \dim\left(\Supp_{++}\left(\frac{\Rees(I_1,\ldots,I_p)}{h\, \Rees(I_1,\ldots,I_p)} \otimes_R R/I_i\right)\right) \\
			&= \dim\left(\Supp_{++}\left(\gr_{I_i}\big(\Rees(I_1,\ldots,\widehat{I_{i}},\ldots,I_p)\big) \otimes_\AAA \AAA/h\AAA\right)\right) \\
			&\le (r+p) - (p+1) -1 = r-2
		\end{align*}
		and 
		$$
		\dim\left(\Supp_{++}\left(\frac{\Rees(I_1,\ldots,I_p)}{h\, \Rees(I_1,\ldots,I_p)}\right)\right) = (r+1+p) - (p+1) -1 = r-1.
		$$
		Finally, the short exact sequence $0 \rightarrow  \Ker(\mathfrak{s}) \rightarrow \frac{\Rees(I_1,\ldots,I_p)}{h\, \Rees(I_1,\ldots,I_p)} \xrightarrow{\mathfrak{s}} \Rees_S(J_1,\ldots,J_p) \rightarrow 0$ and the additivity of mixed multiplicities yield 
		$$
		\deg_\DD^{(d_0,\dd)}(\Gamma) \, = \, e\left((d_0,\dd-\ee_i), \frac{\Rees(I_1,\ldots,I_p)}{h\, \Rees(I_1,\ldots,I_p)}\right) \,=\, e\left((d_0,\dd-\ee_i), \Rees_S(J_1,\ldots,J_p)\right),
		$$
		and so the proof of the proposition is complete.
	\end{proof}
	
	The next lemma will allow us to simplify the families of rational maps we consider in this section. 
	It shows that the mixed multiplicities of $\sfib(I_1,\ldots,I_p)$ depend only on a certain part of the data given and is similar to \cite[Proposition 2.11]{SAT_FIB_GOR_3}.
	
	\begin{lemma}
		\label{lem_mm_prim_ideals}
		Assume \autoref{setup_rat_maps} and suppose that the ideals $I_1,\ldots,I_{p-1}$ are $\mm$-primary. 
		Then, the following statements hold:
		\begin{enumerate}[\rm (i)]
			\item We have the equality
			$$
			\sfib(I_1,\ldots,I_p) \,= \, \bigoplus_{n_1,\ldots,n_p \ge 0} \left[\big(I_p^{n_p}\big)^{\sat}\right]_{n_1\delta_1 + \cdots + n_p\delta_p}.
			$$
			In particular, for all $\dd \in \NN^p$ with $|\dd| = r$, the value of 
			$
			e\big(\dd, \sfib(I_1,\ldots,I_p)\big) \,=\, \deg(\GG) \cdot \deg_\PP^\dd
			(Y)
			$		
			depends only on the degrees $\delta_1,\ldots,\delta_{p-1}$ and on the ideal $I_p$.
			
			\item $
			e\big((0,\ldots,0,r),\, \sfib(I_1,\ldots,I_p)\big) \;=\; e\big((0,r), \, \Rees(I_p)\big) \;=\, e\big(\sfib(I_p)\big).
			$ 
		\end{enumerate} 		
	\end{lemma}
	\begin{proof}
		(i) By the assumption of $I_1,\ldots,I_{p-1}$ being $\mm$-primary ideals, we have $I_1^{n_1}\cdots I_p^{n_p} R_\pp = I_p^{n_p} R_\pp$ for all $\pp \in \Spec(R) \setminus  \{\mm\}$ and $n_1,\ldots,n_p \in \NN$.
		We then obtain the equality $\big(I_1^{n_1}\cdots I_p^{n_p}\big)^\sim = \big(I_p^{n_p}\big)^\sim$ of sheaves.
		It follows that
		$$
		\sfib(I_1,\ldots,I_p) = \bigoplus_{\bn \in \NN^p}  \HH^0\left(\PP_\kk^r, \big(I_1^{n_1}\cdots  I_p^{n_p}\big)^\sim(\bn \cdot \delta)\right)
		= 
		\bigoplus_{\bn \in \NN^p}  \HH^0\left(\PP_\kk^r, \big(I_p^{n_p}\big)^\sim(\bn \cdot \delta)\right) = \bigoplus_{\bn \in \NN^p} \left[\big(I_p^{n_p}\big)^{\sat}\right]_{\bn \cdot \delta}.
		$$
		So, the first statement of the lemma follows. 
		
		(ii) Consider the Rees algebra $\Rees(I_p)$ with the usual bigrading $\Rees(I_p) = \bigoplus_{c,n\in \NN} [\Rees(I_p)]_{c,n}$ where $[\Rees(I_p)]_{c,n} = \big[I_p\big]_{c + n\delta_p}$.
		Let  $Q(c,n) := P_{\Rees(I)}(c,n)$ be the bigraded Hilbert polynomial of $\Rees(I_p)$.
		Choose $e > 0$ such that $\left[\HL^1\left(\Rees(I_p)\right)\right]_k = 0$ for all $k \ge e$.
		From the grading chosen for the Rees algebra, it follows that 
		$$
		\left[\HL^0(R/I_p^n)\right]_{n\delta_p + k} \cong \left[\HL^1(I_p^n)\right]_{n \delta_p + k} =0
		$$ 
		for all $n \ge 0$ and $k \ge e$.
		This implies that  $\left[I_p^n\right]_{n\delta_p + k} = \left[(I_p^n)^\sat\right]_{n\delta_p + k}$ for all $n \ge 0$ and $k \ge e$.
		
		Let $P(\bn) := P_{\sfib(I_1,\ldots,I_p)}(n_1,\ldots,n_p)$ be the multigraded Hilbert polynomial of $\sfib(I_1,\ldots,I_p)$. 
		Choose positive integers $n_1',\ldots,n_p'$ with $n_1'\delta_1+\cdots+n_{p-1}'\delta_{p-1} \ge e$ such that $P(n_1,\ldots,n_p)$ equals the Hilbert function of $\sfib(I_1,\ldots,I_p)$ when $n_i \ge n_i'$.	
		We can also assume that $Q(c, n)$ equals the Hilbert function of $\Rees(I_p)$ when $c \ge c':= n_1'\delta_1+\cdots+n_{p-1}'\delta_{p-1}$ and $n \ge n_p'$.
		Therefore, with the above choices in place, we obtain the following equalities 
		\begin{align*}
			P(n_1',\ldots,n_{p-1}',n) &= 
			\dim_{\kk}\left(\left[\left(I_p^n\right)^\sat\right]_{n_1'\delta_1 + \cdots + n_{p-1}'\delta_{p-1} + n\delta_p}\right)   \\
			&= \dim_{\kk}\left(\left[I_p^n\right]_{n_1'\delta_1 + \cdots + n_{p-1}'\delta_{p-1} + n\delta_p}\right) \\
			&= \dim_{\kk}\left([\Rees(I_p)]_{c', \,n}\right) = Q(c', n)
		\end{align*}
		for all $n \ge n_p'$.
		As a consequence, we obtain
		$$
		e\big((0,\ldots,0,r),\, \sfib(I_1,\ldots,I_p)\big) \,=\,	
		\lim_{n\rightarrow \infty}\, \frac{P(n_1',\ldots,n_{p-1}', n)}{n^r/r!} \,=\, \lim_{n\rightarrow \infty}\, \frac{Q(c',n)}{n^r/r!} \,=\, e\big((0,r), \, \Rees(I_p)\big).
		$$
		Finally,  \cite[Theorem 5.4]{MIXED_MULT} (or \autoref{thm_deg_graph_sat_fib} above) implies that $e\big((0,r),\, \Rees(I_p)\big) = e\big(\sfib(I_p)\big)$, and so we are done.
	\end{proof}

	We are now ready to state our main result in this section. 
	We say that an ideal $I \subset R$ satisfies the condition $G_{r+1}$ when $\mu(I_{\pp}) \le \dim(R_{\pp})$ for all $\pp \in V(I) \subset \Spec(R)$ with $\HT(\pp)<r+1$, where $\mu(I_\pp)$ is the minimal number of generators of $I_\pp$.			
	The following theorem gives exact formulas for the multidegrees of the saturated special fiber ring $e\big(\dd, \, \sfib(I_1,\ldots,I_p)\big)$ for several families of rational maps.
	
	\begin{theorem} \label{thm:multidegree_formula_perfect_ht2_gorenstein_ht3}
		Assume \autoref{setup_rat_maps} and suppose that the ideals $I_1,\ldots,I_{p-1}$ are $\mm$-primary and that $m_p \ge r$. 
		\begin{enumerate}[\rm (I)]
			\item Suppose that the following conditions hold:
			\begin{enumerate}[\rm (a)]
				\item $I_p$ is perfect of height two with Hilbert-Burch resolution of the form
				$$
				0 \rightarrow \bigoplus_{i=1}^{m_p}R(-\delta_p-\mu_i) \rightarrow {R(-\delta_p)}^{m_p+1} \rightarrow I_p \rightarrow 0.			
				$$
				\item $I$ satisfies the condition $G_{r+1}$.
			\end{enumerate}
			Then, for all $\dd = (d_1,\ldots,d_p) \in \NN^{p}$ with $|\dd| = r$, we have 
			$$
			\deg(\GG) \cdot \deg_\PP^\dd(Y) \,=\,  
			\delta_1^{d_1}\cdots \delta_{p-1}^{d_{p-1}}\cdot e_{d_p}(\mu_1,\ldots,\mu_{m_p}),
			$$
			where $	e_{d_p}(\mu_1,\ldots,\mu_{m_p})$ denotes the elementary symmetric polynomial
			$$
			e_{d_p}(\mu_1,\ldots,\mu_{m_p})= 
			\sum_{1\le j_1  < \cdots < j_{d_p} \le m_p} \mu_{j_1}\cdots\mu_{j_{d_p}}.
			$$
			
			\smallskip
			
			\item Suppose that the following conditions hold:
			\begin{enumerate}[\rm (a)]
				\item $I_p$ is a Gorenstein ideal of height three.
				\item Every nonzero entry of an alternating minimal presentation matrix of $I_p$  has degree $D\ge 1$.
				\item $I_p$ satisfies the condition $G_{r+1}$.		
			\end{enumerate}
			Then, for all $\dd = (d_1,\ldots,d_p) \in \NN^{p}$ with $|\dd| = r$, we have 
			$$
			\deg(\GG) \cdot \deg_\PP^\dd(Y) \,=\, \begin{cases}
				\delta_1^{d_1}\cdots \delta_{p-1}^{d_{p-1}}\cdot D^{d_p} \sum_{k=0}^{\lfloor\frac{m_p-d_p}{2}\rfloor}\binom{m_p-1-2k}{d_p-1} \quad \text{ if } d_p \ge 3 \\
				\delta_1^{d_1}\cdots \delta_{p-1}^{d_{p-1}}\delta_p^{d_p} \qquad\qquad \quad\;\qquad\qquad\qquad \text{ otherwise.}
			\end{cases}
			$$	
		\end{enumerate}
	\end{theorem}
	\begin{proof}
		Due to \autoref{thm:main_result_sat_special_fiber_ring}, it suffices to compute the mixed multiplicities $e\big(\dd,\, \sfib(I_1,\ldots,I_{p-1},I_p)\big)$.
		The assumption of $I_1,\ldots,I_p$ being $\mm$-primary allows us to reduce to the case where $I_1=\cdots=I_{p-1} = \mm$.
		By~\autoref{lem_mm_prim_ideals}(i), we have the following relation of multigraded Hilbert polynomials 
		$$
		P_{\sfib(I_1,\ldots,I_{p-1}, I_p)}(n_1,\ldots,n_{p-1}, n_p) \;=\;  P_{\sfib(\mm,\ldots,\mm, I_p)}(n_1\delta_1, \ldots,n_{p-1}\delta_{p-1}, n_p)
		$$		
		for all $n_i \gg 0$.
		This implies that 
		$$
		e\big(\dd,\, \sfib(I_1,\ldots,I_{p-1},I_p)\big) \;=\;  \delta_1^{d_1}\cdots\delta_{p-1}^{d_{p-1}} \cdot e\big(\dd,\,  \sfib(\mm,\ldots,\mm,I_p)\big)
		$$
		for all $\dd = (d_1,\ldots,d_p) \in \NN^{p}$ with $|\dd| = r$.
		Therefore, for the rest of the proof, we assume that $I_1=\cdots =I_{p-1} =\mm$.
		
		(I)  
		We can assume that $\kk$ is an infinite field.
		By \autoref{thm_deg_graph_sat_fib}, we have that $e\big(\dd,\, \sfib(I_1,\ldots,I_{p-1},I_p)\big) = \deg_\DD^{(0,\dd)}(\Gamma)$.
		Since $I_1 = \cdots = I_p = \mm$, by applying \autoref{prop_cut_hyper} successively we have that 
		$$
		\deg_\DD^{(0,d_1,\ldots,d_{p-1},d_p)}(\Gamma) \; = \; e\big((0,0,\ldots,0,d_p), \, \Rees_S(J_1,\ldots,J_{p-1}, J_p)\big).
		$$
		Here $S = R / (l_1,\ldots,l_{s})R$, $s = d_1+\cdots+d_{p-1}$, each $l_i = \alpha_{i,0}x_0 + \cdots + \alpha_{i,r}x_r$ gives a general hyperplane $V(l_i) \subset \PP_\kk^r$, and $J_1 = \cdots = J_{p-1} = \mm S$, $J_p = I_p S$.
		Then \autoref{lem_mm_prim_ideals}(ii) gives 
		$$
		\deg_\DD^{(0,d_1,\ldots,d_{p-1},d_p)}(\Gamma) \; = \; e\big((0,0,\ldots,0,d_p), \, \Rees_S(J_1,\ldots,J_{p-1}, J_p)\big) \;=\; e\big((0,d_p), \, \Rees_S(J_p)\big).
		$$		
		By \cite[Proposition 5.6(i)]{MIXED_MULT}, we have $e\big((0,d_p), \, \Rees_S(J_p)\big) = e\big((s, d_p), \, \Rees(I_p)\big) = d_s(\GG_i)$, where $d_s(\GG_i)$ denotes the $s$-th projective degree of the rational map $\GG_i : \PP_\kk^r \dashrightarrow \PP_\kk^{m_p}$.
		Finally, by applying \cite[Theorem 5.7]{MIXED_MULT}, we obtain the formula 
		$$
		d_s(\GG_i) = e_{d_p}(\mu_1,\ldots,\mu_{m_p}),
		$$
		and so the result follows. 
		
		(II) The proof follows similarly to (I), however we use \cite[Theorem 5.8]{MIXED_MULT} instead of \cite[Theorem 5.7]{MIXED_MULT}.		
	\end{proof}

	\section{Monomial rational maps}\label{sec:monomial_rational_maps}
	
	In this section, we study the case of a monomial rational map. 
	We shall express the mixed multiplicities of the saturated special fiber ring in terms of mixed volumes of some naturally constructed polytopes.
	The following setup is used throughout this section. 
	
	\begin{setup}
		\label{setup_monomial}
		Assume \autoref{setup_rat_maps} and suppose that the polynomials $f_{i,k} \in I_i$ are monomials. 
		So, $I_1, \ldots, I_p \subset R$ are monomial ideals.
	\end{setup}
	
	First, we need to recall some important results and fix our notation. Let $\bK = (K_1,\ldots,K_p)$ be a sequence of convex bodies in $\RR^r$.
	For any sequence $\lambda = (\lambda_1,\ldots,\lambda_p) \in \NN^p$ of nonnegative integers, we denote by $\lambda \bK$ the Minkowski sum $\lambda \bK := \lambda_1 K_1+\cdots+\lambda_p K_p$ and by $\bK_\lambda$ the multiset $\bK_\lambda := \bigcup_{i=1}^p\bigcup_{j=1}^{\lambda_i} \{K_i\}$ of $\lambda_i$ copies of $K_i$ for each $1 \le i \le p$.
	Let us recall Minkowski's classic theorem, see~e.g.~\cite[Theorem 3.2, page 116]{EWALD}.
	
	\begin{theorem}[Minkowski]
		\label{thm_Minkowski}
		$\Vol_{r}(\lambda \bK)$ is a homogeneous polynomial of degree $r$.	
	\end{theorem}
	We write the polynomial $\Vol_{r}(\lambda \bK)$ as
	$$
	\Vol_{r}(\lambda \bK) = \sum_{{\dd \in \NN^{p}\; \lvert \dd \rvert=r}}\; \frac{1}{\dd!}\,\MV_r(\bK_\dd) \, \lambda^\dd,
	$$
	where $\MV_r(-)$ denotes the \emph{mixed volume}.
	
	For a monomial ideal $J \in R$ with monomial generating set $G(J)$, we define the following lattice convex polytope 
	$$
	\Gamma(J) \, := \, \text{ConvexHull}\left( \big\lbrace
	(n_0,n_1,\ldots, n_r) \in \NN^r  \,\mid\, x_0^{n_0}x_1^{n_1}\cdots x_r^{n_r} \in G(J)
	\big\rbrace\right)  \, \subset \, \RR^{r+1}.
	$$
	We denote by $\pi : \RR^{r+1} \twoheadrightarrow \RR^r$ the natural projection given by $(x_0,x_1\ldots,x_r) \mapsto (x_1\ldots,x_r)$.
	We have the following important result that expresses the mixed multiplicities of $I_1,\ldots,I_p$ with respect to $\mm$ in terms of mixed volumes. 
	
	\begin{theorem}[{\cite[Theorem 2.4, Corollary 2.5]{TRUNG_VERMA_MIXED_VOL}}]
		\label{thm_vol_mult}
		Assume \autoref{setup_monomial}.
		Then, for all $\dd  = (d_1,\ldots,d_p) \in \NN^p$ with $|\dd| = r$, we have 
		$$
		e_{(0,\dd)}\left(\mm \mid I_1, \ldots, I_p \right) \,= \, \MV_r\left( {\Big(\pi(\Gamma(I_1)), \, \ldots,\, \pi(\Gamma(I_p))\Big)}_\dd \right).
		$$
	\end{theorem}	
	
	It should be mentioned that this result was generalized in \cite{cid2020convex, cid2021multigraded} for not-necessarily-Noetherian graded families of ideals.
	
	The following theorem is the main result of this section. 
	
	\begin{theorem} \label{thm:monomial_degree_mixedVol}
		Assume \autoref{setup_monomial}.
		Then, for all $\dd = (d_1,\ldots,d_p) \in \NN^p$ with $|\dd| = r$, we have 
		$$
		\deg(\GG) \cdot \deg_\PP^\dd(Y) \,= \, \MV_r\left( {\Big(\pi(\Gamma(I_1)), \, \ldots,\, \pi(\Gamma(I_p))\Big)}_\dd \right).
		$$
	\end{theorem}
	\begin{proof}
		The result follows directly by combining \autoref{thm:main_result_sat_special_fiber_ring}, \autoref{thm_deg_graph_sat_fib}, \autoref{lem_multdeg_mixed_mult} and \autoref{thm_vol_mult}.
		More precisely, we obtain the following sequence of equalities
		\begin{align*}
			\deg(\GG) \cdot \deg_\PP^\dd(Y) &= e\big(\dd, \sfib(I_1,\ldots,I_p)\big) \\
			&= \deg_\DD^{(0,\dd)}(\Gamma)\\
			&= e_{(0,\dd)}\left( \mm \mid I_1, \ldots, I_p \right) \\
			&= \MV_r\left( {\big(\pi(\Gamma(I_1)), \, \ldots,\, \pi(\Gamma(I_p))\big)}_\dd \right)
		\end{align*}
		for all $\dd \in \NN^p$ with $|\dd|=r$.
	\end{proof}

	We now illustrate the above result with three simple but instructive examples.

	\begin{example}\label{example:classical_rational_maps}
		Consider the following classical rational maps
		\[
		\GG_1 : \PP_\kk^2 \dashrightarrow \PP_\kk^2, \quad (x, y, z) \mapsto (x^2, y^2, z^2) \quad
		\textrm{ and }\quad
		\GG_2 : \PP_\kk^2 \dashrightarrow \PP_\kk^2, \quad  (x, y, z) \mapsto (xy, xz, yz) 
		\]
		and denote their corresponding base ideals $I_1 = (x^2, y^2, z^2) $ and $I_2 = (xy, xz, yz) $. Let $\pi$ be the map that projects onto the $x$ and $y$ coordinates. 
		The polytopes $\Gamma(I_1)$ and $\Gamma(I_2)$ are depicted in \autoref{fig:classical_polytope_projections} along with their projections.

		\begin{figure}[H]
			\centering
			\includegraphics[scale  = 0.8]{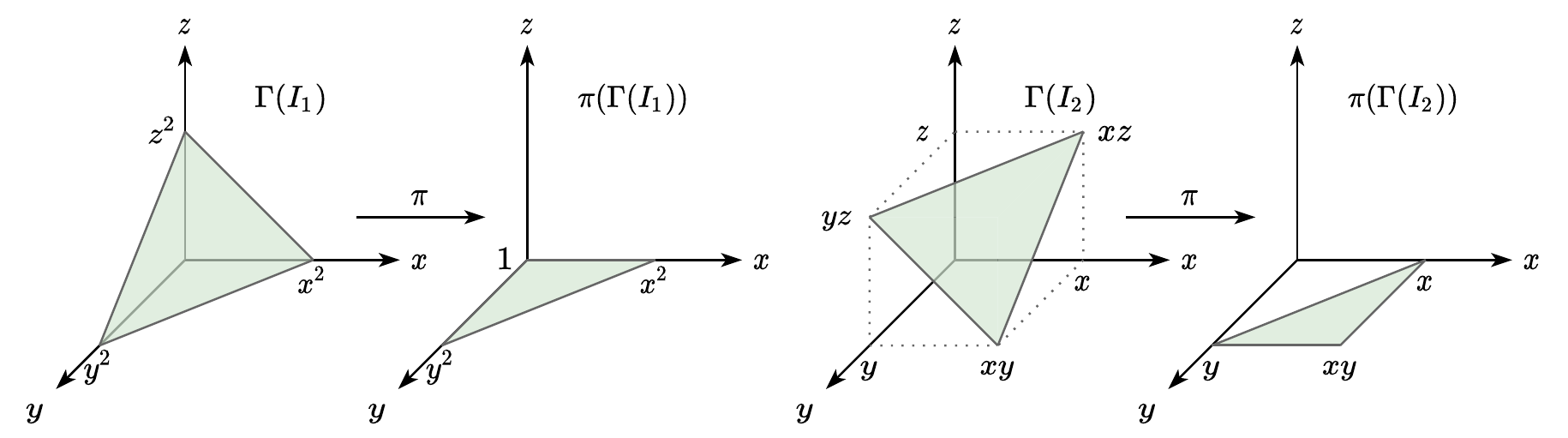}
			\caption{The polytopes associated to the ideals in \autoref{example:classical_rational_maps} and their projections.}
			\label{fig:classical_polytope_projections}
		\end{figure}
		
		\noindent
		Since the image $Y = \PP_\kk^2$ of $\GG_1$ has degree $1$,  by \autoref{thm:monomial_degree_mixedVol}
		we have that
		\[
		\deg(\GG_1) = \MV_2\left(\pi(\Gamma(I_1), \pi(\Gamma(I_1)) = 2! \Vol_2(\pi(\Gamma(I_1))\right) = 4.
		\]
		Similarly, for the rational map $\GG_2$, by  \autoref{thm:monomial_degree_mixedVol} we have
		\[
		\deg(\GG_2) = \MV_2\left(\pi(\Gamma(I_2), \pi(\Gamma(I_2)\right) = 2! \Vol_2(\pi(\Gamma(I_2))) = 1.
		\]
		We note that the equalities $\deg(\GG_1) = 4$ and $\deg(\GG_2) = 1$ are also confirmed by \autoref{thm:zero_base_locus_degree_formula}, since $\GG_1$ has no base points and $\GG_2$ has three base points.
	\end{example}
	
	\begin{example}\label{example_P2}
		Let us consider the following rational map
		\[
		\GG : \PP_\kk^2 \dashrightarrow \PP_\kk^2\times \PP_\kk^2, \qquad  (x, y, z) \mapsto ((x,y,z), (x^3, yz^2, xy^2))
		\]
		which we note has degree $1$. 
		Let $I_1 = (x,y,z)$ and $I_2 = (x^3, yz^2, xy^2)$ be the corresponding base ideals. 
		Let $\pi$ be the map that projects onto the $x$ and $y$ coordinates. 
		In \autoref{fig:P2_polytopes} we depict the polytopes $\pi(\Gamma(I_1))$ and $\pi(\Gamma(I_2))$ along with their Minkowski sum.
		
		\begin{figure}[H]
			\centering
			\includegraphics[scale = 0.8]{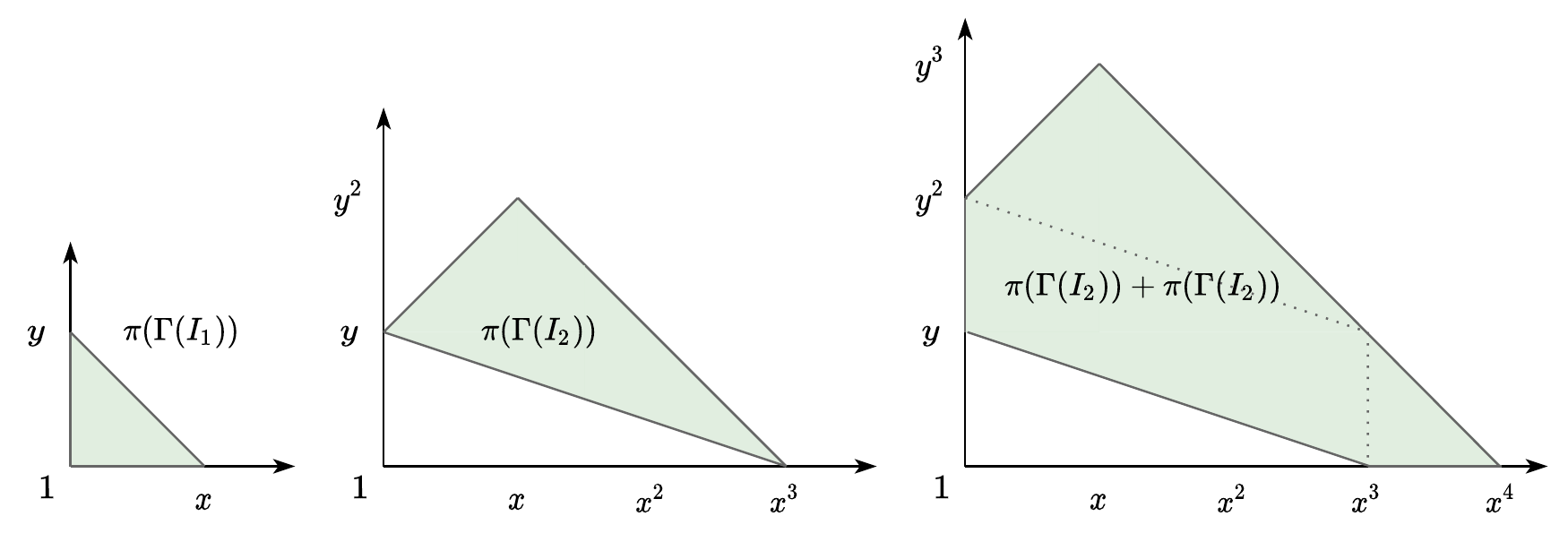}
			\caption{The polytopes $\pi(\Gamma(I_1))$ and $\pi(\Gamma(I_2))$ from \autoref{example_P2} and their Minkowski sum. The depiction of the Minkowski sum (right) contains disjoint copies of the polytopes $\pi(\Gamma(I_1))$ and $\pi(\Gamma(I_2))$ indicated by the faint dotted lines.}
			\label{fig:P2_polytopes}
		\end{figure}
		
		\noindent
		By the definition of the mixed volume, we have \[
		\MV_2\left(\pi(\Gamma(I_1)), \pi(\Gamma(I_2))\right) = 
		\Vol_2\left(\pi(\Gamma(I_1) + \pi(\Gamma(I_2))\right) - \Vol_2\left(\pi(\Gamma(I_1))\right) - \Vol_2\left(\pi(\Gamma(I_2))\right)
		= \frac{11}{2} - \frac{1}{2} - 2 = 3.
		\]
		And so, by \autoref{thm:monomial_degree_mixedVol}, we have that the multidegrees of the image of $\GG$ are given by 
		\[
		\deg_\PP^{(2,0)}(Y) = 1, \quad 
		\deg_\PP^{(0,2)}(Y) = 4 \quad \text{ and } \quad
		\deg_\PP^{(1,1)}(Y) = 3.
		\]
	\end{example}

	\begin{example}\label{example_P3}
		Let us consider the following rational map
		\[
		\GG : \PP_\kk^3 \dashrightarrow \PP_\kk^3\times \PP_\kk^3, \qquad  (x, y, z, w) \mapsto ((x,y,z,w),(xy^2, xw^2, yzw, z^2w)).
		\]
		Let $I_1 = (x,y,z,w)$ and $I_2 = (xy^2, xw^2, yzw, z^2w)$ be the corresponding base ideals. 
		Let $\pi$ be the map that projects onto the $x$, $y$ and $z$ coordinates. 
		The polytopes $\pi(\Gamma(I_1))$ and $\pi(\Gamma(I_2))$ and the Minkowski sums $\pi(\Gamma(I_1)) + \pi(\Gamma(I_2))$ and $2\pi(\Gamma(I_1)) + \pi(\Gamma(I_1))$ are depicted in \autoref{fig:P3_example_polytopes}. 
		For ease of notation, let us write $A = \pi(\Gamma(I_1))$ and $B = \pi(\Gamma(I_2))$. 
		The volumes of the polytopes are
		\[
		\Vol_3(A) = \frac{1}{6},  \quad  \Vol_3(B) = \frac{1}{3}, \quad
		\Vol_3(A + B) = \frac{9}{2} \quad \text{and} \quad
		\Vol_3(2 A + B) = \frac{56}{3}.
		\]

		\begin{figure}[h]
			\centering
			\includegraphics[scale = 0.6]{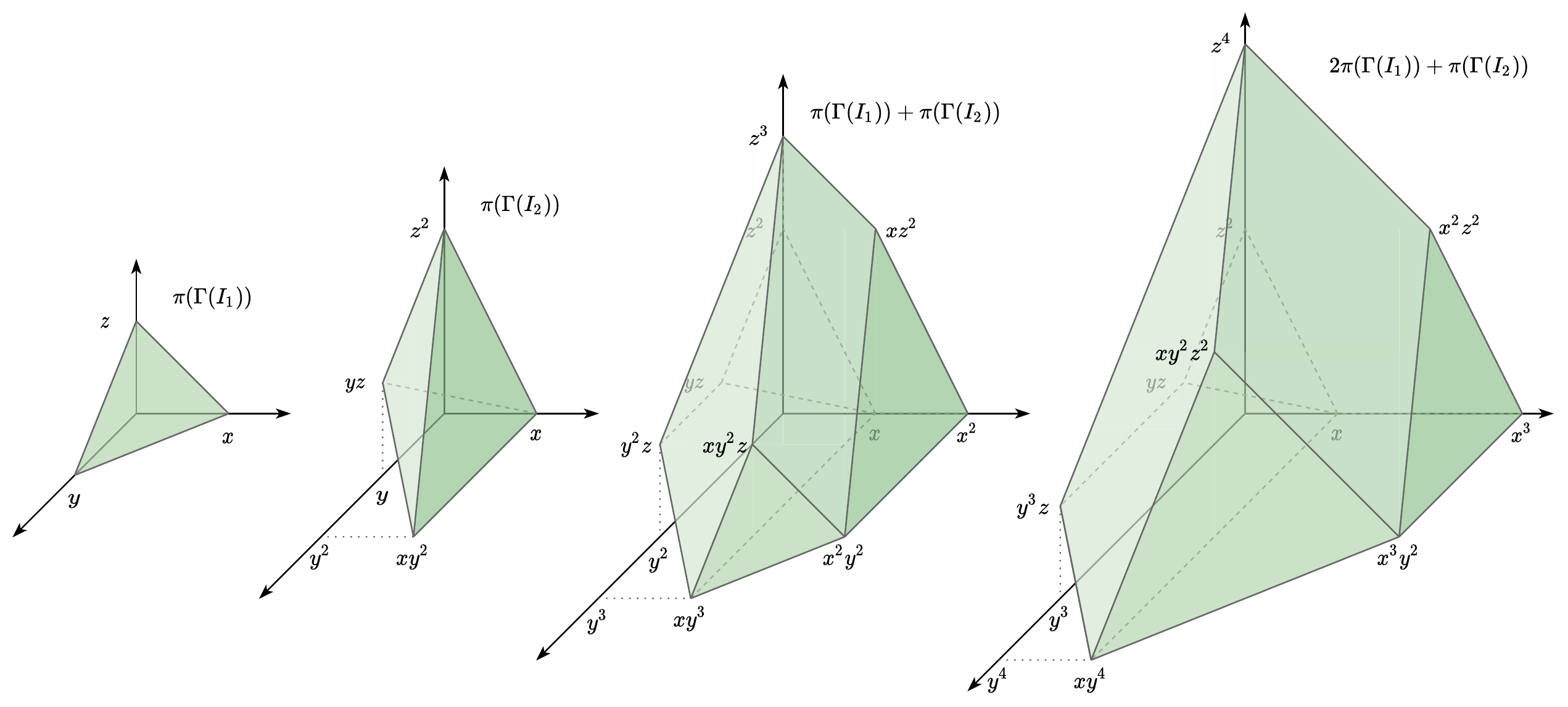}
			\caption{The polytopes $\pi(\Gamma(I_1))$ and $\pi(\Gamma(I_2))$ and their Minkowski sums from \autoref{example_P3}.}
			\label{fig:P3_example_polytopes}
		\end{figure}
		
		\noindent
		By definition of mixed volumes, we have 
		\[
		\Vol_3(A + B) = \Vol_3(A) + \Vol_3(B) +
		\frac 1 2 \MV_3(A, A, B) + 
		\frac 1 2 \MV_3(A, B, B)
		\]
		\[
		\Vol_3(2A + B) = 8 \Vol_3(A) + \Vol_3(B) +
		4 \MV_3(A, A, B) + 
		\MV_3(A, B, B)
		\]
		Solving these equations, we find that $\MV_3(A,A,B) = 3$ and $\MV_3(A,B,B) = 5$. It is straightforward to see that $\deg(\GG) = 1$. And so, by \autoref{thm:monomial_degree_mixedVol} we have 
		\[
		\deg_\PP^{(3,0)}(Y) = 1, \quad
		\deg_\PP^{(2,1)}(Y) = 3, \quad
		\deg_\PP^{(1,2)}(Y) = 5, \quad \text{and} \quad
		\deg_\PP^{(0,3)}(Y) = 2.
		\]

	\end{example}
	
	\section{Application / Interpretation}\label{sec:applications}
	
	The problem of \textit{implicitization} is fundamental in the area of Computer Aided Geometric Design (CAGD). 
	The main problem, in the case of curves, can be stated as follows. Let $a_0,a_1,a_2 \in \kk[t]$ be polynomials, then $(x,y) := (a_1/a_0, a_2/a_0)$ is a rational parametric curve in $\kk^2$. 
	The implicitization of the curve is a polynomial in $x$ and $y$ which carves out this curve. 
	This problem is studied in \cite{COX_EQ_PARAM, cox1998moving_line} using a collection of generators of the syzygy module of $I = (a_0x - a_1, a_0y - a_2) \subseteq \kk[x,y,t]$ called a $\mu$-basis. 
	This setup can be translated into the language used throughout this paper by homogenizing. Let $d = \max\{\deg(a_0), \deg(a_1), \deg(a_2) \}$ and, for each $i \in \{0,1,2\}$, if $a_i = \sum_i \alpha_i t^i$ then define $\overline{a_i} = \sum_i \alpha_i t^i s^{d - i} \in \kk[s,t]$. 
	The parametric curve is the image of the rational map $\GG : \PP_\kk^1 \dashrightarrow \PP_\kk^2$ given by $\GG(s,t) = (\overline{a_0}, \overline{a_1}, \overline{a_2})$, restricted to a particular affine patch. 
	%Assuming that this map is generically finite, we may apply \autoref{thm:zero_base_locus_degree_formula} to determine the degree of the image.
	
	The concept of a $\mu$-basis naturally extends to higher dimensions. 
	For example, $\mu$-bases for surfaces $\GG : \PP_\kk^2 \dashrightarrow \PP_\kk^3 : (s,t) \mapsto (a,b,c,d)$ are studied in \cite[Section~5]{COX_EQ_PARAM}. 
	The $\mu$-basis allows us to determine properties about the parametric surface, i.e.~the image of $\GG$. 
	Under particular assumptions, Cox computes in \cite[Proposition~5.3]{COX_EQ_PARAM} that the degree of the image of $\GG$ can be determined completely in terms of the degrees of a $\mu$-basis. Explicitly, Cox requires that $\deg(\GG) = 1$ and that the base locus $V(a,b,c,d) \subseteq \PP_\kk^2$ of the rational map is a local complete intersection. As a result, the formula for the degree is $\deg_{\PP_\kk^3}(Y) = \mu_1 \mu_2 + \mu_1 \mu_3 + \mu_2 \mu_3 = e_2(\mu_1, \mu_2, \mu_3)$, where the $\mu$-basis of $(a,b,c,d)$ has degrees $\mu_1, \mu_2, \mu_3$. 
	%Note that this follows directly from \autoref{thm:multidegree_formula_perfect_ht2_gorenstein_ht3}.
	In more recent work, \cite[Theorem~A]{MULT_SAT_PERF_HT_2} gives a generalization of this result.  
	It is shown that, for a rational map $\GG : \PP_\kk^r \dashrightarrow \PP_\kk^s$ determined by a perfect base ideal of height two, the degree of the image times the degree of the map equals an elementary symmetric polynomial evaluated at the degrees of the $\mu$-basis. The main tool used in the proof of this formula is the saturated special fiber ring, see \autoref{sec:saturated_special_fiber_ring}.
	
	The study of $\mu$-bases has played an important role in a number of applications. As mentioned above, they appear in the implicitization problem of parametric surfaces and curves.  
	But, they are also an important tool for detecting singularities and in studying the degree and birationality of rational maps. 
	They are nowadays an essential tool in many developments in CAGD, see~e.g.~\cite{survey_mu_basis, CHEN_SEDERB_IMP, CHEN_WANG_YANG_SING, GOLDMAN_JIA, MU_BASIS_RULED_SURF, SEDERBERG20151, SEDERBERG20161}.
	
	More generally, in a somehow parallel story, $\mu$-bases form part of the general syzygy-based approach for studying rational maps. This method appears to have been originally initiated in \cite{HULEK_KATZ_SCHREYER_SYZ}. In this paper, Hulek, Katz and Schreyer give a characterization of when a map is a Cremona transformation in terms of properties of the syzygies. The use of syzygies in studying rational maps has become an active and particularly fruitful research area, see~e.g.~\cite{AB_INITIO, Simis_cremona,KPU_blowup_fibers,EISENBUD_ULRICH_ROW_IDEALS,Hassanzadeh_Simis_Cremona_Sat,SIMIS_RUSSO_BIRAT,EFFECTIVE_BIGRAD,SIMIS_PAN_JONQUIERES,HASSANZADEH_SIMIS_DEGREES, MULTPROJ, MULT_SAT_PERF_HT_2, SAT_FIB_GOR_3, MIXED_MULT, BOTBOL_ALICIA_RAT_SURF, BOTBOL_ALICIA_MAT, KPU_BIGRAD_STRUCT, KPU_GOR3, KPU_NORMAL_SCROLL, CARLOS_MONO, CARLOS_MONOID, CARLOS_MU2}.

	In this paper, we further extend the syzygy-based approach (a.k.a $\mu$-bases) by considering a family of rational maps which has, until this point, been mostly unobserved. 
	We are interested in rational maps from a projective variety to a multiprojective variety, typically, a rational map of the form $\GG : \PP_\kk^r \dashrightarrow \PP_\kk^{m_1} \times_\kk \cdots \times_\kk \PP_\kk^{m_p}$.
	It should be mentioned that rational maps from a multiprojective variety to a projective variety (i.e. a rational map typically of the form $\mathcal{H} : \PP_\kk^{m_1} \times_\kk \cdots \times_\kk \PP_\kk^{m_p} \dashrightarrow \PP_\kk^r $) are an important gadget in the field of Geometric Modeling, see~e.g.~\cite{EFFECTIVE_BIGRAD, SEDERBERG20151, SEDERBERG20161}. 
	We define and use a \emph{multigraded generalization of the saturated special fiber ring} that allows us to obtain important results regarding the \emph{(nonlinear) multiview variety} given as the image of $\GG$.  
	
	Our work is motivated by the study of multiview varieties that were considered in \cite{aholt2013hilbert} and in \cite{agarwal2019ideals}. 
	The classical multiview variety is the image of a linear rational map $\GG : \PP_\kk^3 \dashrightarrow \PP_\kk^2 \times \cdots \times \PP_\kk^2$. 
	Each component of this map $\GG_i : \PP_\kk^3 \dashrightarrow \PP_\kk^2$ is called a \textit{(pinhole) camera} and is given by a full-rank $3 \times 4$ matrix $A_i$. This setup naturally arises in the study of computer vision, and the multiview variety is the closure of the image $Y$ of $\GG$. 
	In general, the defining ideal $J$ of $Y$ is a Cartwright-Sturmfels ideal, and so it has many desirable properties, see~\cite{conca2018cartwright, conca2021radical}.
	The kernel of $A_i$ is a single point in $\PP_\kk^3$ and the base locus of $\GG$ is the collection of these points. 
	In the case that these points are distinct, by \cite[Theorem~3.7]{aholt2013hilbert}, we have an explicit description for the multigraded Hilbert function of $J$. 
	This result follows from explicitly computing the generic initial ideal %$\initGen(J)$ 
	of the defining ideal $J$ of $Y$. 
	Our primary goal is to show that similar relationships hold even in the nonlinear case. 
	In particular, we focus on the highest degree coefficients of the multigraded Hilbert polynomial of $J$, which are the multidegrees of the multiprojective variety $Y$.
	
	To study the ideal of the nonlinear multiview variety, see \autoref{setup_rat_maps}, it is not possible to take the same approach as \cite{aholt2013hilbert} since the generic initial ideals are not radical. 
	Instead, we are able to compute the multidegrees of the image of $\GG$ using the saturated special fiber ring. 
	We show in \autoref{thm:main_result_sat_special_fiber_ring} that our multigraded version of the saturated special fiber ring maintains similar important properties (cf. \cite[Theorem 2.4]{MULTPROJ}).
	Similarly to \cite[Theorem~3.7]{aholt2013hilbert}, we show in \autoref{thm:zero_base_locus_degree_formula} that the multidegrees of $Y$ are related to the mixed multiplicities of the base points.
	In \autoref{sec:linear_rational_maps}, we recover several known results in the setting of the classical linear multiview varieties (see \autoref{thm:multiplicity_free_linear_case}, \autoref{cor_eq_sat_fib_fib}, \autoref{cor:dim_zero_base_linear}). 
	In \autoref{sec:perfect_ht2_gorenstein_ht3}, we provide explicit formulas for several interesting families of rational maps. 
	The techniques used in \autoref{sec:perfect_ht2_gorenstein_ht3} convey the depth and scalability of the saturated special fiber ring to study rational maps. 
	We also prove some technical results that might be interesting in their own right, see \autoref{thm_deg_graph_sat_fib}, \autoref{prop_cut_hyper} and  \autoref{lem_mm_prim_ideals}.
	In \autoref{sec:monomial_rational_maps}, we focus on the case where the rational map $\GG$ is given by monomials, i.e.~the ideals $I_1, \dots, I_p$ are monomial ideals. 
	Using convex geometry, we relate the multidegrees of the image with the mixed volumes of the associated Newton polytopes of the ideals $I_i$ after dehomogenizing. 
	
	\smallskip
	\noindent	\textbf{Future directions.} In this paper, we have computed for large families of rational maps, the multidegrees of their image. 
	Following this, it is natural to ask how we may compute the other coefficients of the multigraded Hilbert polynomial. 
	In particular, we ask whether these coefficients can be determined from the points in the base locus analogously to \autoref{thm:zero_base_locus_degree_formula}.
	Of course, it is also quite desirable to study other families of nonlinear multiview varieties. And so, we ask for which families of multiview varieties is possible to find explicit formulas for their multidegrees.

%	\bibliography{references}

\end{document}